\documentclass[leqno,11pt, a4]{amsart}
\usepackage[utf8]{inputenc}
\usepackage{amsthm}
\usepackage{amsmath}
\usepackage{enumitem}
\usepackage{amsfonts}
\usepackage{amssymb}
\usepackage{graphicx}
\usepackage{pdfsync}
\usepackage{xcolor}
\newtheorem{theorem}{Theorem}[section]
\newtheorem{definition}{Definition} [section]
\newtheorem{corollary}{Corollary}[theorem]
\newtheorem{lemma}[theorem]{Lemma}
\newtheorem{proposition}[theorem]{Proposition}
\newtheorem{remark}[theorem]{Remark}
\newcounter{yuppo}
\newcommand{\Keler} {K\"{a}hler }
\newcommand{\keler} {K\"{a}hler }
\newcommand{\kelerian} {K\"{a}hlerian }

\newcommand{\End}{\operatorname{End}}

\newcommand{\cd}{\cdot}
\newcommand{\om}{\omega}

\renewcommand{\phi}{\varphi}

\newcommand{\ra}{\rightarrow}
\newcommand{\lra}{\longrightarrow}
\newcommand{\C}{\mathbb{C}}
\newcommand{\R}{\mathbb{R}}

\newcommand{\Gl}{\operatorname{Gl}}

\newcommand{\Ad}{\operatorname{Ad}}

\newcommand{\ga}{\gamma}

\newcommand{\enf}{\emph}

\newcommand{\desudt}[1][]{\dfrac {\mathrm {d} #1 }{\mathrm {dt}}}
\newcommand{\desudtzero}{\desudt \bigg \vert _{t=0} }
\newcommand{\zero}{\bigg \vert _{t=0} }

\newcommand{\liu}{\mathfrak{u}}

\newcommand{\lia}{\mathfrak{a}}

\newcommand{\liek}{\mathfrak{k}}
\newcommand{\lier}{\mathfrak{r}}

\newcommand{\lieg}{\mathfrak{g}}

\newcommand{\liep}{\mathfrak{p}}

\newcommand{\liez}{\mathfrak{z}}

\newcommand{\la}{\lambda}

\newcommand{\alfa}{\alpha}
 
\newcommand{\sx}{\langle} 
\newcommand{\xs}{\rangle}

\newcommand{\spam}{\operatorname{span}}

\newcommand{\liea}{\mathfrak{a}}

\newcommand{\noparty}[1]{}

\newcommand{\scalo}{\sx \cdot , \cdot \xs}

\newcommand{\mup}{\mu_\liep}
\newcommand{\mupb}{\mu_\liep^\beta}
\title{Stability, analytic stability for real reductive Lie groups }
\author{Leonardo Biliotti}
\address{Dipartimento di Scienze Matematiche, Fisiche e Informatiche \\
          Universit\`a di Parma (Italy)}
\email{leonardo.biliotti@unipr.it}
\author{Oluwagbenga Joshua Windare}
\address{Dipartimento di Scienze Matematiche, Fisiche e Informatiche \\
          Universit\`a di Parma (Italy)}
\email{oluwagbengajoshua.windare@unipr.it}
\thanks{The first author was partially supported by PRIN  2017
   ``Real and Complex Manifolds: Topology, Geometry and holomorphic dynamics ''
   and ``National Group for Algebraic and Geometric Structures, and their Applications'' (GNSAGA - INDAM)}
\begin{document}
\maketitle
\begin{abstract}
\noindent
We presented a systematic treatment of a Hilbert criterion for stability theory for an action of a real reductive group $G$ on a real submanifold $X$ of a \Keler manifold $Z$. More precisely, we suppose that the action of a compact Lie group with Lie algebra $\liu$ extends holomorphically to an action of the complexified group $U^\C$ and that the $U$-action on $Z$ is Hamiltonian. If  $G\subset U^\C$ is compatible, there is a corresponding gradient map $\mu_\mathfrak{p} : X\to \mathfrak{p}$, where $\lieg = \liek \oplus \liep$ is a Cartan decomposition of the Lie algebra of $G$. The concept of energy complete action of $G$ on $X$ is introduced. For such actions, one can characterize stability, semistability and polystability of a point by a numerical criteria using a $G$-equivariant function called maximal weight. We also prove the classical Hilbert-Mumford criteria for semistability and polystability conditions.
\end{abstract}
\section{Introduction}
\pagenumbering{arabic}
In this paper, we study the Hilbert Mumford criterion for semistability and polystability points associated with the actions of real reductive groups on real submanifolds of \Keler manifolds. In the classical case of a group action on a K\"ahler manifold these characterization are given by Mundet i Riera \cite{MUNDETC,MUNDETT}, Bruasse and Teleman \cite{BT}, Teleman \cite{Teleman}, Kapovich, Leeb
and Millson \cite{KLM} and probably many others, indeed any of these ideas go
back as far as Mumford \cite[\S 2.2]{MUMFORD}, where systematic presentation of notions of stability theory in the non-algebraic \kelerian geometry of complex reductive Lie groups and the relationships between these notions are given. Our aim is to investigate a class of actions of real reductive Lie groups on real submanifolds of a \keler manifold using gradient map techniques. More precisely, we consider a \Keler manifold $(Z,\omega)$ with an holomorphic action of a complex reductive Lie group $U^\C$, where $U^\C$ is the complexification of a compact  Lie group $U$ with Lie algebra $\mathfrak{u}$. We also assume $\omega$ is $U$-invariant and that there is a $U$-equivariant momentum map $\mu : Z \to \mathfrak{u}^*.$ By definition, for any $\xi \in \mathfrak{u}$ and $z\in Z,$ $d\mu^\xi = i_{\xi_Z}\omega,$ where $\mu^\xi(z) := \mu(z)( \xi )$ and $\xi_Z$ denotes the fundamental vector field induced on $Z$ by the action of $U,$ i.e.,
$$\xi_Z(z) := \frac{d}{dt}\zero \exp(t\xi)z
$$
(see, for example, \cite{Kirwan} for more details on the momentum map).
Recently, the momentum map has been generalized to the following settings \cite{PG,heinzner-schwarz-stoetzel,heinzner-schuetzdeller}.

We say that a subgroup $G$ of $U^\C$ is compatible if $G$ is closed and the map $K\times \mathfrak{p} \to G,$ $(k,\beta) \mapsto k\exp(\beta)$
is a diffeomorpism where $K := G\cap U$ and $\mathfrak{p} := \mathfrak{g}\cap \textbf{i}\mathfrak{u};$ $\mathfrak{g}$ is the Lie algebra of $G$.
The Lie algebra $\mathfrak{u}^\C$ of $U^\C$ is the direct sum $\mathfrak{u}\oplus \textbf{i}\mathfrak{u}.$ It follows that $G$ is compatible with the
Cartan decomposition of $U^\C = U\exp(\textbf{i}\mathfrak{u})$, $K$ is a maximal compact subgroup of $G$ with Lie algebra $\mathfrak{k}$ and that
$\mathfrak{g} = \mathfrak{k}\oplus \mathfrak{p}$ (see Section \ref{comp-subgrous}). 
The inclusion $\textbf{i}\mathfrak{p}\hookrightarrow \mathfrak{u}$ induces by restriction, a $K$-equivariant map $\mu_{\textbf{i}\mathfrak{p}} : Z \to (\textbf{i}\mathfrak{p})^*.$ One can choose and fix $B$ an $\mathrm{Ad}(U^\C)$-invariant inner product of the Euclidian type on the Lie algebra $\liu^\C$, see \cite[Section 3.2]{BT}, \cite[Definition 3.2.4]{LT1} and also \cite[Section 2.1]{He} for the analogue in the algebraic GIT. Such an inner product will automatically induce a well defined inner product on any maximal compact subgroup $U'$ of $U^\C$. We thank the referee for pointing out these references.

Let $\scalo$ denote the real part $B$. Then $\scalo$ is positive definite on $\textbf{i}\liu$, negative definite on $\liu$, $\langle \liu, \textbf{i} \liu \rangle=0$ and finally the multiplication by $\textbf{i}$ satisfies $\langle \textbf i \cdot ,\textbf{i} \cdot \rangle=-\scalo$ (see Subsection \ref{sgm}).  In order to simplify the notation we replace consideration of $\mu_{\text{i}\mathfrak{p}}$ by that of $\mup:Z \lra \liep$, where
$$\mu_\mathfrak{p}^\beta (x):=\langle \mup(x),\beta \rangle:=\langle \textbf {i}\mu(x),\beta\rangle =-\langle \mu(x),-\textbf{i}\beta\rangle=\mu^{-\textbf{i} \beta}(x).$$
The map $\mup:Z \lra \liep$ is $K$-equivariant and grad$\, \mu_\mathfrak{p}^\beta = \beta_Z$ for any $\beta \in \liep$. Here the grad is computed with respect to the Riemannian metric induced by the \Keler structure. The map $\mu_\mathfrak{p}$ is called the $G$-gradient map associated with $\mu$ (see Section \ref{subsection-gradient-moment}).
For a $G$-stable locally closed real submanifold $X$ of $Z,$ we consider $\mu_\mathfrak{p}$ as a map $\mu_\mathfrak{p} : X\to \mathfrak{p}$ such that $\mathrm{grad}\, \mup=\beta_X$, where the gradient is now computed with respect to the induced Riemannian metric on $X$. The norm square of $\mu_\mathfrak{p}$ is defined as
$f(x) := \frac{1}{2}\parallel \mu_\mathfrak{p}(x)\parallel^2; \quad x \in X.$ A strategy for analyzing the $G$-action on $X$ is to view $f$ as a generalized Morse function. In \cite{heinzner-schwarz-stoetzel} the authors proved that associated to critical points of $f$, we have $G$-invariant submanifolds of $X$ called pre-strata and they studied its properties. If $X$ is compact then the pre-strata are the strata of a Morse type stratification of $X$.
 Semistable (polystable) points in $X$ can be identified by taking into account the positions of their $G$-orbits with respect to
 $\mu_\liep^{-1}(0).$ A point $x\in X$ is polystable if it's $G$-orbit intersects the level set $\mu_\liep^{-1}(0)$
 ($G\cdot x \cap \mu_\liep^{-1}(0) \neq \emptyset$) and semistable if the closure of it's $G$-orbit intersects
 $\mu_\liep^{-1}(0)$ ($\overline{G\cdot x} \cap \mu_\liep^{-1}(0) \neq \emptyset$).
 The set of semistable points associated with the critical points of $f$ is studied in great details  \cite{PG,heinzner-schwarz-stoetzel,heinzner-stoetzel}.  The aim of these papers were to develop a Geometrical Invariant Theory for actions of real Lie groups on real submanifolds of a K\"ahler manifold, generalizing results of Heinzner-Huckleberry-Loose \cite{HGI,HH,HHL,HL} and Sjamaar \cite{sjamaar}.


On the other hand, the semistability (polystability) condition can be checked using a numerical criterion in term of a function called the maximal weight, which can be regarded as a \kelerian version of the Hilbert Criterion in Geometric Invariant Theory \cite{MUNDETT,Teleman}. We refer to this numerical condition as analytic stability condition. Biliotti and Zedda \cite{LM}, see also \cite{bgs}, give a systematic treatment of the stability theory for an action of a real non-compact reductive Lie group $G$ on a compact connected real manifold $X$ using the maximal weight function and apply this settings to the action of $G$ on measures of $X.$

In this paper, we introduce a large class of actions of $G$ on $X$ following Teleman \cite{Teleman}, called \emph{energy complete action}. As Teleman pointed out, the energy completeness condition gives the natural framework for the stability theory in non-algebraic complex geometry. For such actions,
we prove that the semistability condition, respectively polystability condition, is  equivalent to analytic semistability condition (Theorem \ref{semistable}), respectively  analytic polystability condition (Theorem \ref{polystable}), extending the results due to Teleman \cite{Teleman} and Biliotti and Zedda \cite{LM}, respectively \cite[Theorem 7.4]{Salamon}. We also characterize the semistable and polystable points in terms of one-parameter subgroups (Corollary \ref{hmss} and Corollary \ref{hmps}).

The paper is organized as follows:

Section \ref{prelim} contains a few basic facts concerning real reductive Lie groups and the gradient map. Many of the results in this section are well known but we include short proofs in some cases. We also recall the  Kempf-Ness function associated to the $G$-action on $X$ and some of its properties are given. 

In Section \ref{analytic}, we define the stability conditions and the maximal weight associated with the gradient map and some of its properties are given. We also introduce and discuss the class of energy complete actions and we prove the moment-weight inequality (Theorem \ref{GMWI}).

In Section \ref{final-section} we introduce the notions of analytically semistability and polystability and we prove the main results.

In Section \ref{linear} we discuss examples which illustrate the general theory.

In Section \ref{final}, following Teleman \cite{Teleman}, we introduce the notion of symplectization of the $U^\C$-action on $Z$ with respect to $G$.
The notions of stability, polystability and semistability depend of the Cartan decomposition of $G$ induced by the Cartan decomposition of $U^\C=U\exp(\textbf{i}\liu)$ and so these concepts depend on the triple $(U,\mathtt g,\mu)$, where $\mathtt g$ is the Riemannian metric induced by the K\"ahler form  $\omega$. The norm square of the gradient map also depends on the triple $(U,\mathtt g,\mu)$.

Let  $g\in G$. Then $U'=gUg^{-1}$ is a maximal compact subgroup $U^\C$ and $U^\C=U'\exp(\textbf{i}\liu')$. $U'$ preserves the K\"ahler form $\omega_g=(g^{-1})^{\star}\omega$ and there exists a momentum map $\mu:Z \lra \liu'$, given by $\mu'=\mathrm{Ad}(g)\circ \mu \circ g^{-1}$. Note that $(g^{-1})^{\star}\mathtt g$ is the Riemannian metric associated to $\omega_g$. Now, $G$ is also compatible with $U^\C=U'\exp(\textbf{i}\liu')$. Indeed,   $G=K'\exp(\liep')$, where $K'=gKg^{-1}$ and $\liep'=\mathrm{Ad}(g)(\liep)$. Moreover,  $\mathrm{Ad}(g)\circ \mup \circ g^{-1}$ is the $G$-gradient map associated to $\mu'$.  This means that one may also chose the triple $(U',(g^{-1})^{\star}\mathtt g,\mu')$ for analyzing the $G$-action on $X$. We show that the polystable and semistable conditions and the stratification of $X$ with respect to the norm square gradient map do not depend of the triple chosen.

{\bfseries \noindent{Acknowledgements.}}   We wish to thank Alessandro Ghigi and Peter Heinzner  for interesting discussions. We would also like to thank the anonymous referee for carefully reading our paper and for giving such constructive comments which substantially helped improving the quality of the paper.
\section{Preliminaries}\label{prelim}
\subsection{Compatible subgroups and parabolic subgroups}
\label{comp-subgrous}
Let $H$ be a Lie group with Lie algebra $\mathfrak{h}$ and $E, F \subset \mathfrak{h}$. Then, we set
 \begin{gather*}
   E^F :=     \{\eta\in E: [\eta, \xi ] =0, \forall \xi \in F\} \\
   H^F = \{g\in H: \Ad g (\xi ) = \xi, \forall \xi \in F\}.
 \end{gather*}
 If $F=\{\beta\}$ we write simply $E^\beta$ and $H^\beta$.

 Let $U$ be a compact Lie group and let $U^\C$ be the corresponding complex linear algebraic group \cite{chevalley}. The group $U^\C$ is reductive and is the universal complexification of $U$ in the sense of \cite{ho}. On the Lie algebra level, we have the Cartan decomposition $\liu^\C=\liu\oplus \textbf{i} \liu$ with a corresponding Cartan involution $\theta:\liu^\C \lra \liu^\C$ given by $\xi+\textbf{i}\nu \mapsto \xi-\textbf{i}\nu$. We also denote by $\theta$ the corresponding involution on $U^\C$. The real analytic map $F:U\times \textbf{i}\liu \lra U^\C$, $(u,\xi) \mapsto u\exp(\xi)$ is a diffeomorphism. We refer to the composition $U^\C=U\exp(\textbf{i}\liu)$ as the Cartan decomposition of $U^\C$.

 Let $G\subset U^\C$ be a closed real
 subgroup of $U^\C$. We say
 that $G$ is \enf{compatible} with the Cartan decomposition of $U^\C$ if $F (K \times \liep) = G$ where $K:=G\cap U$ and $\liep:= \lieg \cap \textbf{i}\liu$. The
 restriction of $F$ to $K\times \liep$ is then a diffeomorphism onto
 $G$. It follows that $K$ is a maximal compact subgroup of $G$ and
 that $\lieg = \liek \oplus \liep$.  Since $K$ is a retraction of $G$, it follows that $G$ has only finitely many connected components and $G=KG^o$, where $G^o$ denotes the connected component of $G$ containing $e$.  Since $U$ can be embedded in $\Gl(N,\C)$ for
 some $N$, and any such embedding induces a closed embedding of
 $U^\C$, any compatible subgroup is a closed linear group. By \cite[Proposition 1.59 p.57]{knapp-beyond},
 $\lieg$ is a real reductive Lie algebra, and so $\lieg =
 \liez(\lieg)\oplus [\lieg, \lieg]$, where $\liez(\lieg)=\{v\in \lieg:\, [v,\lieg]=0\}$ is the Lie algebra of the center of $G$. Denote by $G_{ss}$ the analytic
 subgroup tangent to $[\lieg, \lieg]$. Then $G_{ss}$ is closed and
 $G^o=Z(G^o)^o\cd G_{ss}$ \cite[p. 442]{knapp-beyond}, where $Z(G^o)^o$ denotes the connected component of the center of $G^o$ containing $e$.
 \begin{lemma}[\protect{\cite[Lemma 7]{LA}}]\label{lemcomp}$\ $ \begin{enumerate}
   \item \label {lemcomp1} If $G\subset U^\C$ is a compatible
     subgroup, and $H\subset G$ is closed and $\theta$-invariant,
     then $H$ is compatible if and only if $H$ has only finitely many connected components.
   \item \label {lemcomp2} If $G\subset U^\C$ is a connected
     compatible subgroup, then $G_{ss}$ is compatible.
     \item \label{lemcomp3} If $G\subset U^\C$ is a compatible
       subgroup and $E\subset \liep$ is any subset, then $G^E$ is
       compatible. Indeed, $G^E=K^E\exp(\liep^E)$, where $K^E=K\cap G^E$ and $\liep^E=\{x\in \liep:\, [x,E]=0\}$.
   \end{enumerate}
 \end{lemma}
If $\beta \in \liep$ the endomorphism $\mathrm{ad}(\beta) \in\mathrm{End}(\lieg)$ is diagonalizable over $\R$. Denote by $V_\lambda (\mathrm{ad}(\beta))$ the eigenspace of $\mathrm{ad}(\beta)$ corresponding to the eigenvalue $\lambda$. We define,
\begin{align*}
  G^{\beta+} &:=\{g \in G : \lim_{t\to - \infty} \exp (t\beta)\, g
  \exp (-t\beta) \text { exists} \}\\
  R^{\beta+} &:=\{g \in G : \lim_{t\to - \infty} \exp (t\beta)\,
  g \exp (-t\beta) =e \}\\
    \lier^{\beta+}: = \bigoplus_{\la > 0} V_\la (\mathrm{ad}( \beta)) \\
  G^{\beta-} &:\{g\in G:\, \lim_{t\to + \infty} \exp (t\beta)\, g \exp (-t\beta) \text { exists}\}\\
   R^{\beta-} &:=\{g \in G : \lim_{t\to + \infty} \exp (t\beta)\, g \exp(-t\beta)=e \}\\
  \lier^{\beta-}: = \bigoplus_{\la < 0} V_\la (\mathrm{ad}(\beta)).
\end{align*}
Note that $G^{\beta-}=G^{(-\beta)+}=\theta(G^{\beta+})$, $\theta(R^{\beta+})=R^{\beta-}$ and $G^{\beta}=G^{\beta+}\cap G^{\beta-}$.
The group $G^{\beta+}$, respectively $G^{\beta-}$, is a parabolic subgroup of $G$ with unipotent radical $R^{\beta+}$, respectively $R^{\beta-}$ and Levi factor $G^{\beta}$.
$R^{\beta+}$ is connected with Lie algebra $\lier^{\beta+}$. Hence $R^{\beta-}$ is connected with Lie algebra $\lier^{\beta-}$.
 The parabolic subgroup $G^{\beta+}$ is the semidirect product of $G^\beta$ with $R^{\beta+}$ and we have the projection $\pi^{\beta+} : G^{\beta+} \lra G^\beta$,
 $
 \pi^{\beta+}(g)=\lim_{t\mapsto -\infty} exp(t\beta)g\exp(-t\beta).
 $
Analogously, $G^{\beta-}$ is the semidirect product of $G^{\beta}$ with $R^{\beta-}$ and we have the projection $\pi^{\beta-}:G^{\beta-} \lra G^\beta$,  $\pi^{\beta-}(g)=\lim_{t\mapsto +\infty} exp(t\beta)g\exp(-t\beta)$. In particular, $ \lieg^{\beta+}= \lieg^\beta \oplus \lier^{\beta+}$, respectively  $ \lieg^{\beta-}= \lieg^\beta \oplus \lier^{\beta-}$.
\begin{proposition}\label{parabolic-decomposition}
For any $\beta \in \liep$, we have $G=KG^{\beta+}$.
\end{proposition}
\begin{proof}
If $G$ is connected, the result is well-known, see for instance \cite[Lemma 9]{LA} and \cite[Lemma 4.1]{heinzner-schwarz-stoetzel}. Since $G=KG^o$, it follows that
$
G=KG^o=K(G^{o})^{\beta+}=KG^{\beta+},
$
concluding the proof.
\end{proof}
\subsection{Gradient map}
\label{subsection-gradient-moment}\label{sgm}
Let $(Z, \om)$ be a \Keler manifold on which $U^\C$ acts
holomorphically on $Z$. Assume that $U$ preserves $\om$ and that there is a $U$-equivariant
momentum map $\mu: Z \ra \liu^*$. By definition we then have $\mathrm{d} \mu^\xi=i_{\xi_Z} \omega$, where $\mu^\xi(z):=\mu(z)(\xi)$ and $\displaystyle \xi_Z(z) := \frac{d}{dt}\zero \exp(t\xi)z$. We choose and fix an $\mathrm{Ad}(U^\C)$ inner product of Euclidian type on the Lie algebra $\liu^\C$ \cite[Section 3.2]{BT} and \cite[Definition 3.2.4]{LT12}. One can use for instance an embedding of $U$ in $\mathrm{U}(n)$ and its extension $U^\C \lra \mathrm{GL}(n,\C)$. Then
\[
B:\liu^\C \times \liu^\C \lra \C, \qquad (X,Y) \mapsto \mathrm{Tr}(XY)
\]
is an $\mathrm{Ad}(U^\C)$ inner product of Euclidian type on the Lie algebra $\liu^\C$.

Let $\scalo$ denote the real part of $B$. Then, keeping in mind that $\liu \subset \mathrm{u}(n)$, $\scalo$ is positive-definite on $\textbf{i}\liu$, negative-define on $\liu$, $\langle \liu, \textbf{i} \liu \rangle=0$ and $\langle \textbf{i}\cdot , \textbf{i} \cdot \rangle =-\scalo$. In particular we may identify $\liu$ and $\liu^*$ by means of $-\scalo$ and so we may of think the momentum map as $\liu$-valued map.

Let $G \subset U^\C$ be a compatible subgroup of $U^\C$. Then we have $\liep\subset \textbf{i}\liu$. If $z\in Z$, then the orthogonal projection of $\textbf{i}\mu(z)$ onto $\liep$ with respect to $\scalo$ defines a $K$-equivariant map $\mup:Z \lra \liep$. In other words, we define $\mup$ requiring that for any $\beta \in \liep$, we have
\[
\mup^\beta(z):=\langle \mup(z), \beta \rangle =\langle \textbf{i}\mu(z),\beta\rangle=-\langle \mu(z),-\textbf{i}\beta\rangle=\mu^{-\textbf{i}\beta}(z).
\]
The map $\mup:Z \lra \liep$ is called the $G$-gradient map associated with $\mu$. The equation $\mathrm{d} \mu^\xi =i_{\xi_Z} \omega$ is equivalent to $\mathrm{grad}\, \mup^\beta =\beta_Z$ for any $\beta \in \liep$. The gradient is computed with respect to the Riemannian structure given by the K\"ahler form $\omega$ on $Z$, i.e., $(v, w) = \om (v, Jw)$.
For the rest of this paper, we fix a $G$-invariant locally closed real submanifold $X$ of $Z$ and we also denote the restriction of $\mu_\mathfrak{p}$ to $X$ by $\mu_\mathfrak{p}.$ Then  $\mup:X \lra \liep$ is a $K$-equivariant map such that $
\text{grad}\mu_\mathfrak{p}^\beta = \beta_X
$. Here the gradient is now computed with respect to the induced Riemannian metric on $X$ that we also denote by $(\cdot,\cdot)$. Similarly, $\bot$ denotes perpendicularity relative to the Riemannian metric on $X.$
\begin{remark}\label{splitting}
If $\lieg\subset \liu^\C$ is compatible with respect to the Cartan decomposition $\liu^\C=\liu \oplus \textbf{i}\liu$, then $\scalo_{\lieg \times \lieg}$ is non-degenerate. Hence $\liu^\C=\lieg\oplus \lieg^\perp$, where $\lieg^\perp=\{v\in \liu^\C:\, \langle v,\lieg\rangle=0\}$.
\end{remark}
Let $\lia \subset \liep$ be an Abelian subalgebra. Then the $A=\exp(\lia)$-gradient map on $X$ is given by  $\mu_{\lia}=\pi_\lia \circ \mup$, where  $\pi_\lia:\liep \lra \lia$ denotes the orthogonal projection of $\liep$ onto $\lia$. We recall some of the properties of the gradient map.

For any subspace $\mathfrak{m}$ of $\mathfrak{\lieg}$ and $x\in X,$ let $$\mathfrak{m}\cdot x := \{\xi_X(x) : \xi \in \mathfrak{m}\}$$ and let $\liep_x:=\{\xi \in \liep:\, \xi_X (x)=0\}$. We denote by $G_x$ and $K_x$ the stabilizer subgroup of $x\in X$ with respect to the $G$-action and the $K$-action respectively and by $\lieg_x$ and $\liek_x$ their respective Lie algebras.
\begin{lemma}
Let $x\in X$. Then $\text{Ker}\,\, d\mu_\mathfrak{p}(x) = (\mathfrak{p}\cdot x)^\bot $.
\end{lemma}
\begin{proof}
This follows from the basic fact that $\mathrm{grad}\, \mup^\beta=\beta_X$ for any $\beta \in \liep$.
\end{proof}
As a corollary, one can prove the following result.
\begin{corollary}
Let $x\in X.$ The following are equivalent:
\begin{enumerate}
    \item $\mathrm (d\mu_\mathfrak{p})_x : T_x X \to \mathfrak{p}$ is onto.
    \item $(\mathrm d \mu_\mathfrak{p} )_x : \mathfrak{p}\cdot x \to \mathfrak{p}$ is onto.
    \item $\liep_x=\{0\}$.
\end{enumerate}
\end{corollary}

Since $\mathrm{grad}\, \mup^\beta (x)=\beta_X (x)$, the following result is easy to check.
\begin{lemma}\label{increasing}
Let $x\in X$ and let $\beta \in \mathfrak{p}.$ Then either $\beta_X(x) = 0$ or the function $t\mapsto \mu_\mathfrak{p}^\beta(\exp(t\beta) x)$ is strictly increasing.
\end{lemma}
\begin{corollary}\label{lem}
If $x \in X$ and $\beta \in \mathfrak{p},$ then $\mu_\mathfrak{p}(\exp(\beta) x) = \mu_\mathfrak{p}(x)$ if and only if $\beta_X(x) = 0.$
\end{corollary}
\begin{proof}
Suppose $\mu_\mathfrak{p}(\exp(\beta) x) = \mu_\mathfrak{p}(x).$ Let $y(t)=\mu_\mathfrak{p}^\beta(\exp(t\beta) x)$. By the above Lemma,  $y(0)=y(1)=0$ if and only if $\beta_X(x)=0$, concluding the proof.
\end{proof}
\begin{proposition}\label{stabilizer of a stable point}
 Let $x\in X$ be such that $\mup(x)=0$. Then
\begin{enumerate}
\item $G\cdot x \cap \mup^{-1}(0)=K\cdot x$.
\item $G_x=K_x \exp (\liep_x)$.
\end{enumerate}
In particular $G_x$ is a compatible subgroup of $G$.
\end{proposition}
\begin{proof}
Let $g=k\exp(\xi)\in G$, where $k\in K$ and $\xi \in \liep$, be such that $gx \in \mup^{-1}(0)$. By the $K$-equivariance of $\mup$, it follows that $\exp(\xi)x\in \mup^{-1}(0)$. By the above Corollary, $\exp(\xi)x=x$ and so $gx=kx$. For the second part just note that $k\exp(\xi)x=x$ implies $\exp(\xi)x=x$. By the above Corollary $\xi \in \liep_x$ and so $k\in K_x$. Hence $G_x=K_x \exp(\liep_x)$ is compatible.
\end{proof}
\begin{theorem}\label{line}[Slice Theorem \protect{\cite[Thm. 3.1]{heinzner-schwarz-stoetzel}}]
If $x \in X$ and $\mup(x) = 0$, there are a $G_x$-invariant
  decomposition $T_x X = \lieg \cd x \oplus W$, open $G_x$-invariant
  subsets $S \subset W$, $\Omega \subset X$ and a $G$-equivariant
  diffeomorphism $\Psi : G \times^{G_x}S \ra \Omega$, such that $0\in
  S, x\in \Omega$ and $\Psi ([e, 0]) =x$.
\end{theorem}
Here $G \times^{G_x}S$ denotes the associated bundle with principal
bundle $G \ra G/G_x$.
\begin{corollary} \label{slice-cor} If $x \in X$ and $\mup(x) = \beta$,
  there are a $G^\beta$-invariant decomposition $T_x X = \lieg^\beta
  \cd x \, \oplus W$, open $G^\beta$-invariant subsets $S \subset W$,
  $\Omega \subset X$ and a $G^\beta$-equivariant diffeomorphism $\Psi
  : G^\beta \times^{G_x}S \ra \Omega$, such that $0\in S, x\in \Omega$
  and $\Psi ([e, 0]) =x$.
\end{corollary}
This follows applying the previous theorem to the action of $G^\beta$ on $X$. Indeed, by Lemma \ref{lemcomp} $G^\beta=K^\beta\exp(\liep^\beta)$ is compatible and the orthogonal projection of $\textbf{i} \mu$ onto $\liep^\beta$ is the $G^\beta$-gradient map $\mu_{\liep^\beta}$. The group $G^\beta$ is also compatible with the Cartan decomposition of $(U^\C )^{\beta}=(U^\C )^{\textbf{i} \beta}=(U^{\textbf{i}\beta} )^\C$ and $\textbf{i} \beta$ is fixed by the $U^{\textbf{i}\beta}$-action on $\liu^{\textbf{i}\beta}$. This implies that $\widehat{\mu_{\liu^{\textbf{i} \beta}}}:Z \lra \liu^{\textbf{i} \beta}$
given by $\widehat{\mu_{\liu^{\textbf{i}\beta}}(z)}=\pi_{\liu^{\textbf{i}\beta}} \circ \mu+\textbf{i}\beta$, where $\pi_{\liu^{\textbf{i}\beta}}$ is the orthogonal projection of $\liu$ onto  $\liu^{\textbf{i}\beta}$, is the $U^{\textbf{i}\beta}$-shifted momentum map. The associated $G^{\beta}$-gradient map is given by $\widehat{\mu_{\liep^\beta}} := \mu_{\liep^\beta} -
\beta$. Hence, if $G$ is commutative, then we have a Slice Theorem for $G$ at every point of $X$,
see \cite[p.$169$]{heinzner-schwarz-stoetzel} and \cite{sjamaar} for more details.

If $\beta \in \liep$, then $\beta_X$ is a vector field on
$X$, i.e. a section of the bundle $TX$. For $x\in X$, the differential is a map
$T_x X \ra T_{\beta_X (x)}(TX)$. If $\beta_X (x) =0$, there is a canonical splitting $T_{\beta_X (x)}(TX) = T_x X \oplus
T_x X$. Accordingly the differential of $\beta_X$, regarded as a section of $TX$, splits into a horizontal and a vertical part. The horizontal part is the identity map. We denote the
vertical part by $\mathrm d \beta_X  (x)$. The linear map $\mathrm d \beta_X  (x)\in \End(T_x X)$ is indeed the so-called intrinsic differential of $\beta_X$, regarded as a section in the tangent bundle $TX$, at the vanishing point $x$. Let
$\{\phi_t=\exp(t\beta)\} $ be the flow of $\beta_X$.  There
  is a corresponding flow on $TX$. Since $\phi_t(x)=x$, the flow on
  $TX$ preserves $T_x X$ and there it is given by $d\phi_t(x) \in
  \Gl(T_x X)$.  Thus we get a linear $\R$-action on $T_x X$ with
  infinitesimal generator $d\beta_X  (x) $.
\begin{corollary}\label{slice-cor-2}
  If $\beta \in \liep $ and $x \in X$ is a critical point of $\mupb$,
  then there are open invariant neighborhoods $S \subset T_x X$ and
  $\Omega \subset X$ and an $\R$-equivariant diffeomorphism $\Psi : S
  \ra \Omega$, such that $0\in S, x\in \Omega$, $\Psi ( 0) =x$. Here
  $t\in \R$ acts as $d\phi_t(x)$ on $S$ and as $\phi_t$ on $\Omega$.)
\end{corollary}
\begin{proof}
 Since $\exp:\liep \lra G$ is a diffeomorphism onto the image, the subgroup $H:=\exp(\R \beta)$ is closed and so it is compatible.  Hence, it is enough to apply the previous corollary to the $H$-action on $X$ and the value at  $x$ of the corresponding gradient map.
\end{proof}
Let $\beta \in \mathfrak{p}$ and let $X^\beta =\{z\in X:\, \beta_X(z)=0\}.$ By Corollary \ref{slice-cor-2}, $X^\beta$ is a smooth, possibly disconnected, submanifold of $X$.

Let $x \in X^\beta$. Let $D^2\mup^\beta(x) $ denote the Hessian,
which is a symmetric operator on $T_x X$ such that
\begin{gather*}
  ( D^2 \mup^\beta(x) (v, v) = \frac{\mathrm d^2}{\mathrm
    dt^2} 
  (\mup^\beta\circ \ga)(0)
\end{gather*}
where $\ga$ is a smooth curve, $\ga(0) = x$ and $ \dot{\ga}(0)=v$. Using the inner scalar product $(\cdot,\cdot)(x)$ one obtains an associate symmetric endomorphism $A(\mup^\beta) (x)$ of the Euclidian vector space $(T_x X, (\cdot,\cdot)(x))$.
Denote by $V_-$ (respectively $V_+$) the sum of the eigenspaces of $A(\mup^\beta) (x)$ corresponding to negative (resp. positive)
eigenvalues. Denote by $V_0$ the kernel.  Since $A(\mup^\beta) (x)$ is a
symmetric endomorphism, we get an orthogonal decomposition
\begin{gather}
  \label{Dec-tangente}
  T_x X = V_- \oplus V_0 \oplus V_+.
\end{gather}
Let $\alfa : G \ra X$ be the orbit map: $\alfa(g) :=gx$.  The
differential $d\alfa_e$ is the map $\xi \mapsto \xi_X (x)$. The following result is well-know. A proof is given in \cite{LA}.
\begin{proposition}\label{linearization1}
  If $\beta \in \liep$ and $x \in  X^\beta$ then
  \begin{gather*}
    A(\mup^\beta)(x) = \mathrm d \beta_X (x).
  \end{gather*}
  Moreover $d\alfa_e (\mathfrak r^{\beta\pm} ) \subset V_\pm$ and $d\alfa_e(
  \lieg^\beta) \subset V_0$.  If $G$ acts transitively on $X$, then these are
  equalities.
\end{proposition}
\begin{corollary}
  \label{MorseBott}
  For every $\beta \in \liep$, $\mupb$ is a Morse-Bott function.
\end{corollary}
\begin{proof}
By Corollary \ref{slice-cor-2}, we have $T_x X^\beta = V_0$ for $x\in X^\beta$. Hence, the
  first statement of Proposition \ref{linearization1} shows that the Hessian
  is non-degenerate in the normal directions.
\end{proof}
\begin{proposition}\label{invariance-critical points}
$X^\beta$ is  $G^\beta$-invariant. Moreover, if $x\in X^{\beta}$ and $g\in G^\beta$, then $\mup^\beta (x)=\mup^\beta(gx)$.
\end{proposition}
\begin{proof}
Let $x\in X^\beta$ and let $g\in G^\beta$. Since $\beta_X(gx)=(\mathrm d g)_{x}(\beta_X(x))$, it follows that $gx\in X^\beta$.
By Lemma \ref{lemcomp}, $G^\beta=K^\beta \exp(\liep^\beta)$. Hence $g=k\exp(\xi)$, where $k\in K^\beta$ and $\xi \in \liep^\beta$. By the $K$-equivariance of the $G$-gradient map we have $\mup^\beta (gx)=\mup^\beta(\exp(\xi)x)$. Let $y(t)=\mup^\beta (\exp(t\xi)x)$. Then
\[
\dot{y} (t)=(\beta_X (\exp(t\xi)x), \beta_X (\exp(t\xi)x) ).
\]
By the first part of the proof, $\beta_X (\exp(\xi)x)=0$, and so $
\dot{y} (t)=0$. This implies $\mup^\beta (gx)=\mup^\beta(\exp(\xi)x)=\mup^\beta (x)$, concluding the proof.
\end{proof}
\begin{proposition}\label{gradient-restriction}
The restriction $(\mup)_{|_{X^\beta}}$ takes value on $\liep^\beta$ and so it coincides with= $(\mu_{\liep^\beta})_{|_{X^\beta}}$.
\end{proposition}
\begin{proof}
If $G=U^\C$ then the proof is easy to check. Indeed, let $x\in Z^\beta=Z^{\textbf{i}\beta}$. Since $\exp(t\textbf{i}\beta)x=x$, by the $U$-equivariance of the momentum map, we have
\[
\mu(x)=\mu(\exp(t\textbf{i}\beta)x)=\mathrm{Ad}(\exp(t\textbf{i}\beta))(\mu(x)).
\]
Taking the limit, we get $[\mu(x),i\beta]=0$. In the general case, it suffices to prove that $\textbf{i}\mu(x) \in \textbf{i}\liu^\beta$ implies that its orthogonal projection onto $\liep$ belongs to $\liep^\beta$.

Since $\beta \in \lieg$, it follows that $\mathrm{ad}(\beta)$ leaves $\lieg$ invariant. By remark \ref{splitting}, it follows that $\liu^\C=\lieg \oplus \lieg^\perp$ with respect to $\scalo$ and so $\mathrm{ad}(\beta)$ leaves also invariant $\lieg^\perp$. Moreover, keeping in mind that $\lieg=\liek\oplus \liep$ is an orthogonal decomposition with respect to $\scalo$, it follows that
\[
\lieg^\perp=\liek^\perp \cap \liu \oplus \liep^\perp \cap \textbf{i} \liu.
\]
Suppose $\textbf{i}\mu(x)=a+b \in \textbf{i}\liu^\beta$, where $a \in \liep$ and $b\in \liep^\perp \cap \textbf{i}\liu$. Then $b\in \lieg^\perp$ and
\[
0=[\beta,\textbf{i}\mu(x)]=[\beta,a]+[\beta,b],
\]
and so $[\beta,a]\in \lieg$ and $[\beta,b]\in \lieg^\perp$. This implies $[\beta,b]=0$ concluding the proof. We thank the referee for pointing out this
short argument.
\end{proof}
\subsection{The Kempf-Ness Function}\label{Kempf-Ness Function}
Given $G$ a real reductive group which acts smoothly on $Z;$ $G = K\text{exp}(\mathfrak{p}),$ where $K$ is a maximal compact subgroup of $G.$ Let $X$ be a $G$-invariant locally closed submanifold of $Z.$ As Mundet pointed out in \cite{MUNDET}, there exists a function $\Phi : X \times G \rightarrow \mathbb{R},$ such that
$$
\langle \mu_\mathfrak{p}(x), \xi\rangle = \desudtzero \Phi (x, \exp(t\xi)), \qquad \xi \in \mathfrak{p},
$$ and satisfying the following conditions:

\begin{enumerate}
    \item For any $x\in X,$ the function $\Phi (x, .)$ is smooth on $G.$
    \item The function $\Phi(x, .)$ is left-invariant with respect to $K,$ i.e., $\Phi(x, kg) = \Phi(x, g).$
    \item For any $x\in X,$ $v\in \mathfrak{p}$ and $t\in \mathbb{R};$

    $$\frac{d^2}{dt^2}\Phi (x, \exp(tv)) \geq 0.$$
    Moreover: $$\frac{d^2}{dt^2}\Phi (x, \exp(tv)) = 0$$ if and only if $\exp(\mathbb{R}v)\subset G_x.$

    \item For any $x\in X,$ and any $g, h \in G;$
    $$\Phi(x, hg) = \Phi(x, g) + \Phi(gx, h).$$ This equation is called the cocycle condition. Finally, using the cocycle condition, we have
\[
\desudt \Phi(x,\exp(t\beta))=\langle \mup(\exp(t\xi)x),\beta\rangle.
\]
\end{enumerate}
The function $\Phi : X \times G \rightarrow \mathbb{R}$ is called the Kempf-Ness function for $(X, G, K)$. It is just the restriction of the classical Kempf-Ness function $Z\times U^\C \lra \R$ considered in \cite{MUNDET, Teleman} to $X\times G$  \cite{LM}.

Let $M = G/K$ and $\pi : G \rightarrow M.$ $M$ is a symmetric space of non-compact type \cite{helgason}. Let $x\in X$. By property $(b),$ the function $g\mapsto \Phi(x,g^{-1})$  descends to a function on $M$:
$$\Phi_x : M \lra \mathbb{R}, \qquad \Phi_x (gK):=\Phi(x,g^{-1}).$$
\begin{lemma}\cite{Properties of Gradient map}\label{differential}
For $x\in X$, the differential of $\Phi_x :M \lra \R$ is given as
$$
d(\Phi_x)_{\pi(g)}(v_x) = -\langle \mu_\mathfrak{p}(g^{-1}x), \xi\rangle,
$$ where, $v_x(g) = (\mathrm d \pi)_{g} \left( (\mathrm d L_g)_e (\xi) \right)$ and $\xi \in \mathfrak{p}$.
Therefore, $\nabla \Phi_x(\pi(g)) = -(\mathrm d \pi)_{g} \left(dL_g(\mu_\mathfrak{p}(g^{-1}x))\right),$ where the grad is computed with respect to the $G$-invariant Riemannian metric on $M$ induced by $\scalo$ restricted to $\liep$.
\end{lemma}
\begin{theorem}\cite{Properties of Gradient map}\label{ConvexKempf}
Assume that $X$ is compact. Let $x\in X$ and let $\Phi_x : M \to \mathbb{R}.$ Then
\begin{enumerate}
    \item $\Phi_x$ is a Morse-Bott function and it is convex along geodesics.
    \item If $\gamma : \mathbb{R}\to M$ is a negative gradient flow of $\Phi_x,$ then, $$\lim_{t\to \infty}\Phi_x(\gamma(t)) = \text{inf}_{x\in M}\Phi_x.$$
\end{enumerate}
\end{theorem}
\section{Stability and Maximal Weight Function}\label{analytic}
Let $(Z,\omega)$ be a \Keler manifold and $U^\C$ acts holomorphically on $Z$ with a momentum map $\mu : Z \to \mathfrak{u}.$ Let $G\subset U^\C$ be a closed compatible subgroup. $G = K\exp(\mathfrak{p}),$ where $K := G\cap U$ is a maximal compact subgroup of $G$ and $\mathfrak{p} := \mathfrak{g}\cap \text{i}\mathfrak{u};$ $\mathfrak{g}$ is the Lie algebra of $G.$

Suppose $X\subset Z$ is a $G$-stable locally closed connected real submanifold of $Z$ with the gradient map $\mu_\mathfrak{p} : X\to \mathfrak{p}.$ We recall that by $G_x$ and $K_x$, we denote the stabilizer subgroup of $x\in X$ with respect to the $G$-action and the $K$-action respectively and by $\lieg_x$ and $\liek_x$ their respective Lie algebras.
\begin{definition}
Let $x\in X.$ Then:
\begin{enumerate}
\item $x$ is stable if $G\cdot x \cap \mu_\liep^{-1}(0) \neq \emptyset$ and $\lieg_x$ is conjugate to a Lie subalgebra of $\liek.$
\item $x$ is polystable if $G\cdot x \cap \mu_\liep^{-1}(0) \neq \emptyset.$
\item $x$ is semistable if $\overline{G\cdot x} \cap \mu_\liep^{-1}(0) \neq \emptyset.$
\end{enumerate}
\end{definition}
We denote by $X^s_{\mup}$, $X^{ss}_{\mup}$, $X^{ps}_{\mup}$ the set of stable, respectively semistable, polystable, points.
It follows directly from the definitions above that the conditions are $G$-invariant in the sense that if a point satisfies one of the conditions, then every point in its orbit satisfies the same condition, and for stability, recall that $\lieg_{gx} = \text{Ad}(g)(\lieg_x).$

One may define a relation $\sim$ on $X^{ss}_{\mup}$ where $x\sim y$ if $\mup^{-1}(0)\cap \overline{G\cdot x}\cap \overline{G\cdot y} \neq \emptyset$. This relation is indeed an equivalence relation \cite{PG} and we denote the corresponding quotient by $X^{ss}_{\mup}//G$ and call it the topological Hilbert quotient of $X^{ss}_{\mup}$ by the action of $G$. Let $\pi:X^{ss}_{\mup} \lra X^{ss}_{\mup}//G$ denote the quotient map. The results of Heinzner-Schwarz-St\"otzel \cite[Quotient Theorem p.164]{heinzner-schwarz-stoetzel} and \cite{PG,heinzner-stoetzel}, show the following theorem.
\begin{theorem}
The subsets $X_{\mup}^s$, $X_{\mup}^{ss}$ are open subset in $X$. Moreover, the topological Hilbert quotient $X^{ss}_{\mup}//G$ has the following properties.
\begin{enumerate}
\item Every fiber contains a unique closed $G$-orbit. Any other orbit in the fiber has strictly larger  dimension.
\item The closure of every $G$-orbit in a fiber of $\pi$ contains the closed $G$-orbit.
\item every fiber of $\pi$ intersects $\mup^{-1}(0)$ in a unique $K$-orbit which lies in the unique closed $G$-orbit.
\item the inclusion $\mup^{-1}(0) \hookrightarrow X$ induces a homeomorphism $\mup^{-1}(0) /K \cong X^{ss}_{\mup} //G$.
\end{enumerate}
\end{theorem}
Therefore $X^{ss}_{\mup}//G$ can be identified with the space of polystable orbits. On the other hand the set of polystable points is in general neither open nor closed.
The following result establishes a relation between the Kempf-Ness function and the polystability condition.
\begin{proposition}\label{polystable-I}
Let $x\in X$ and let $g\in G$. The following conditions are equivalent:
\begin{enumerate}
\item $\mup(gx)=0$.
\item $g$ is a critical point of $\Phi(x,\cdot)$.
\item $g^{-1}K$ is a critical point of $\Phi_x$.
\end{enumerate}
\end{proposition}
\begin{proof}
Lemma \ref{differential} proves $(a)$ is equivalent to $(c)$. Since $\Phi(x,\cdot)$ is $K$-invariant it follows that $(b)$ is equivalent to $(c)$, concluding the proof.
\end{proof}
\begin{proposition}
Let $x\in X$.
\begin{itemize}
\item If $x$ is polystable, then $G_x$ is reductive.
\item If $x$  is stable, then $G_x$ is compact.
\end{itemize}
\end{proposition}
\begin{proof}
Assume $\mup(x)=0$. By Lemma \ref{stabilizer of a stable point}, $G_x$ is compatible. Therefore if $x$ is stable, respectively polystable, then $G_x$ is compact, respectively reductive.
Since $G_{gx}=gG_xg^{-1}$, the result follows.
\end{proof}
\subsubsection{Maximal Weight Function}
In this section, we introduce the numerical invariants $\la(x,\beta)$ associated to an element $x\in X$ and $\beta \in \mathfrak{p}.$

For any $t\in \mathbb{R},$ define $\la(x,\beta,t) = \langle\mu_\mathfrak{p}(\exp(t\beta)x), \beta\rangle.$
$$
\la(x,\beta,t) = \langle\mu_\mathfrak{p}(\exp(t\beta)x), \beta\rangle = \frac{d}{dt}\Phi(x, \exp(t\beta)),
$$ where $\Phi: X\times G\to \mathbb{R}$ is the Kempf-Ness function. By the properties of the Kempf-Ness function,
$$
\frac{d}{dt}\la(x,\beta,t) = \frac{d^2}{dt^2}\Phi(x, \exp(t\beta)) \geq 0.
$$
This means that $\la(x,\beta,t)$ is a non decreasing function as a function of $t.$

\begin{definition}
The maximal weight of $x\in X$ in the direction of $\beta \in \mathfrak{p}$ is the numerical value
$$
\la(x,\beta) = \lim_{t\to \infty}\la(x,\beta,t)\in \mathbb{R}\cup\{\infty\}.
$$
\end{definition}
From the proof of Lemma \ref{increasing} we have
$$\frac{d}{dt}\la(x,\beta,t) = \parallel\beta_X(\exp(t\beta) x)\parallel^2,
$$ 
and so,
$$
\la(x,\beta,t) =  \langle\mu_\mathfrak{p}(x), \beta\rangle + \int_0^t\parallel\beta_X(\exp(s\beta) x)\parallel^2 \mathrm{ds}.
$$
For any $x\in X$ and $\beta \in \liep$, we consider the curve $c_x^\beta:[0,+\infty) \lra X$ defined by $c_x^\beta (t)=\exp(t\beta)x$. The energy functional of the curve  $c_x^\beta$ is given by
$$
E(c_x^\beta) = \int_0^{+\infty}\parallel \beta_X(\exp(t\beta)x) \parallel^2 \mathrm{dt}.
$$
Thus,
\begin{equation}\label{energy}
\la(x,\beta) = \la(x,\beta, 0) + E(c_x^\beta),
\end{equation}
\begin{lemma}\label{function}
 Let $x\in X$ and let $\beta\in \mathfrak{p}$. Then the function $(0,\infty) \to \mathbb{R}:$ $t\mapsto t^{-1}\Phi(x,\exp(t\beta))$ is nondecreasing and
$$
\la(x, \beta) = \lim_{t\to \infty}\frac{\Phi (x,\exp(t\beta))}{t}.
$$
\end{lemma}
\begin{proof}
Since
$$
\frac{d}{dt}\Phi(x,\exp(t\beta)) = \la(x,\beta,t),
$$
we have
$$
\Phi(x,\exp(t\beta)) = \int_0^t\la(x,\beta,s)ds.
$$ Since $\la(x,\beta,s)$ is nondecreasing and by definition of maximal weight, we have
$$
\la(x,\beta) = \lim_{t\to\infty}\frac{1}{t}\int_0^t\la(x,\beta,s)ds = \lim_{t\to\infty}\frac{\Phi(x,\exp(t\beta))}{t}.
$$
That the function $t\mapsto t^{-1}\Phi(x,\exp(t\beta))$ is nondecreasing follows from the fact that the function $t\mapsto \Phi(x,\exp(t\beta))$ is convex.
\end{proof}
The following Lemma will be needed \cite{Teleman}. We include the proof for completeness.
\begin{lemma}\label{kempfNess properness}
Let $V$ be a subspace of $\liep.$ The following are equivalent for a point $x\in X:$
\begin{enumerate}
    \item The map $\Phi(x, \cdot)$ is linearly proper on $V$, i.e. there exist positive constants $C_1$ and $C_2$ such that:
    $$
    \parallel v \parallel\leq C_1\Phi(x,\exp(v)) + C_2, \quad \forall v \in V.
    $$
    \item $\lambda(x,\beta) > 0$ $\forall \beta \in V\backslash \{0\}.$
\end{enumerate}
\end{lemma}
\begin{proof}
(a) implies (b): By the inequality in $(a)$, for any $\beta \in V$, $t\in \mathbb{R},$ $t\parallel\beta\parallel \leq C_1\Phi(x,\exp(t\beta)) + C_2$. This shows that
$$
\frac{d}{dt}\Phi(x,\exp(t\beta))=  \lambda(x, \beta, t)>0,
$$
for $\beta \neq 0$ and sufficiently large $t$. Since the function $s\mapsto \lambda(x,\beta,s)$ is nondecreasing, keeping in mind the definition of $\lambda(x,\beta)$, (b) holds.

Suppose $\lambda(x, \beta) > 0, \forall \beta\in V\backslash \{0\}.$  Suppose by contradiction that there are no positive constants $(C_1, C_2)$ for which the inequality in (a) holds. Then that will implies that there is a sequence $(\beta_n)_n$ in $V$ such that
\begin{equation}\label{Inequality}
n\Phi(x,\exp(\beta_n)) + n^2 < \parallel \beta_n \parallel.
\end{equation}

Observe that $\lim_{n\to \infty}\parallel \beta_n \parallel = \infty,$ otherwise, $(\beta_n)_n$ would have a bounded subsequence $(\beta_{n_m})_m$. It will then follow that $(\Phi(x,\exp(\beta_{n_m}))_m$ would also be bounded which will contradict (\ref{Inequality}).

From (\ref{Inequality}), we have
$$
\frac{\Phi(x,\exp(\beta_n))}{\parallel \beta_n \parallel} + \frac{n^2}{\parallel \beta_n \parallel}< \frac{1}{n}.
$$ In particular,
$$
\frac{\Phi(x,\exp(\beta_n))}{\parallel \beta_n \parallel} < \frac{1}{n}.
$$
Set $a_n = \parallel \beta_n \parallel,$ $b_n = \frac{\beta_n}{\parallel \beta_n \parallel},$ and choose $t_0\in \mathbb{R}.$ By the convexity property of $\Phi$,
\begin{align*}
&\Phi(x,\exp(a_nb_n)) \geq \Phi(x,\exp(t_0b_n)) + (a_n - t_0)\frac{d}{dt}\bigg \vert _{t=0} \Phi(x,\exp(tb_n))\\
\implies &\Phi(x,\exp(\beta_n)) \geq \Phi(x,\exp(t_0b_n)) + (a_n - t_0)\lambda(x, b_n, t_0) \quad \forall a_n \geq t_0\\
\implies & \frac{\Phi(x,\exp(t_0b_n)) + (a_n - t_0)\lambda(x, b_n, t_0)}{\parallel \beta_n \parallel} \leq \frac{\Phi(x,\exp(\beta_n))}{\parallel \beta_n \parallel} < \frac{1}{n}.
\end{align*}
We get,
\begin{equation}\label{equation}
    \frac{\Phi(x,\exp(t_0b_n))}{\parallel \beta_n \parallel} + \left(1 - \frac{t_0}{\parallel \beta_n \parallel}\right)\lambda(x, b_n, t_0) < \frac{1}{n}
\end{equation}
The sequence $(b_n)_n$ has a subsequence which converges to, say $b_0\in V$ with $\parallel b_0\parallel = 1.$ Taking the limit of (\ref{equation}), we have $\lambda(x, b_0, t_0)\leq 0.$ But this implies that $\lambda(x,b_0) \leq 0,$ which contradicts the assumption.
\end{proof}
We conclude this section proving a numerical criterium for the stability condition. We start with the following Lemma.
\begin{lemma}\label{stable-polystable}
Let $x\in X$ be such that $\mup(x)=0$.  Then $\lambda(x,\beta) \geq 0$ for any $\beta \in \liep$ and $\lambda(x,\beta)=0$ if and only if $\beta_X(x)=0$. In particular, $x$ is stable if and only if $\lambda(x,\beta)>0$ for any $\beta \in \liep \backslash\{0\}$.
\end{lemma}
\begin{proof}
 Let $\beta \in \liep$  and let
\[
\rho(t):\R \lra \R, \qquad t\mapsto \Phi(x,\exp(t\beta)).
\]
$\rho$ is a convex function and
\[
\lambda(x,\beta) \geq \rho'(t)=\langle \mup(\exp(t\beta) x), \beta \rangle.
\]
Since $\rho'(0)=\langle \mup(x),\beta \rangle =0$, keeping in mind that $\rho'(t)$ is nondecreasing, it follows that $\lambda(x,\beta)\geq 0$ and $\lambda(x,\beta)=0$ if and only if $\rho'(t)=0$ and hence if and only if $\beta_X (x)=0$.
By Proposition \ref{stabilizer of a stable point}, $G_x$ is compatible. Hence $x$ is stable if and only if $G_x\subset K$ and so if and only if $\lambda(x,\beta) >0$ for any $\beta \in \liep \backslash\{0\}$.
\end{proof}
\begin{theorem}\label{stable}
Let $x\in X$. Then $x$ is stable if and only if $\lambda(x,\beta) >0$ for any $\beta \in \liep\backslash\{0\}$.
\end{theorem}
\begin{proof}
Assume $\lambda(x,\beta) >0$ for any $\beta \in \liep\backslash\{0\}$. By Lemma \ref{kempfNess properness}, keeping in mind that $\exp:\liep \lra G/K$ is a diffeomorphism, it follows that $\Phi_x$ is an exhaustion. Hence $\Phi_x$ has a minimum and so a critical point. By Proposition \ref{polystable-I}, $x$ is polystable. Hence there exists $g\in G$ such that $\mup(gx)=0$. We claim that $\Phi(gx,\cdot)$ is linearly properly on $\liep$. Indeed, by cocycle condition, we get
\[
\Phi(gx,\exp(\xi))=\Phi(x,\exp(\xi)g)-\Phi(x,g).
\]
Write $\exp(\xi)g=k(\xi)\exp(\theta (\xi))$. Then $\Phi(gx,\exp(\xi))=\Phi(x,\exp (\theta(\xi))-\Phi(x,g)$. In \cite{MUNDETC}, the author proves that there exist $A_1$ and $A_2$ such that $\parallel \xi \parallel^2 \leq A_1 \parallel \theta(\xi) \parallel^2 +A_2$. Since $\Phi(x,\cdot)$ is linearly proper on $\liep$, it follows that $\Phi(gx,\cdot)$ is linearly proper on $\liep$ as well. By Lemma \ref{kempfNess properness} and Lemma \ref{stable-polystable}, it follows that $gx$ is stable and so $x$ is stable as well.

Assume that $x$ is stable. Then there exists $g\in G$ such that $\mup(gx)=0$. By Lemma \ref{stable-polystable}, $\lambda(gx,\beta) >0$ for any $\beta \in \liep\backslash\{0\}$. By Lemma \ref{kempfNess properness}, $\Phi(gx,\cdot)$ is linearly proper on $\liep$. As in the previous part of the proof, one has $\Phi(x,\cdot)$ is linearly proper on $\liep$ and so $\lambda(x,\beta)>0$ for any $\beta \in \liep\backslash\{0\}$.
\end{proof}
\subsection{Energy Complete}
Let $\mup:X \lra \liep$ denote the $G$-gradient map associated with the momentum map $\mu:Z \lra \liu$.
\begin{definition}
The $G$-action on $X$ is called energy complete if
 for any $x\in X$ and for any $\beta \in \mathfrak{p}$, If $E(c_x^\beta ) < \infty$ then $\lim_{t\to \infty}\exp(t\beta)x$ exists.
\end{definition}
By Corollary \ref{slice-cor-2}, if there exists a sequence $\{t_n\}_{n\in \mathbb N}$ such that $\lim_{n\mapsto +\infty} t_n =\infty$ and $\lim_{n\mapsto \infty}\exp(t_n \beta)x$ converges then the curve $c_x(t)=\exp(t\beta)x$ has a limit as $t\mapsto \infty$. The proof of the following result is similar to \cite[Proposition 3.9]{Teleman}.
\begin{proposition}$\,$
\begin{enumerate}
\item If $X$ is compact, then the $G$-action is energy complete;
\item if $G$ acts on a complex vector space $(V,h)$, where $h$ is a $K$-invariant Hermitian scalar product, then the $G$-action is energy complete;
\item if $G\subset \mathrm{SL}(n,\R)$ is closed and compatible, then the $G$-action on $\R^n$ is energy complete
\end{enumerate}
\end{proposition}
\begin{proposition}\label{energy-complete}
Let $x\in X$ and let $\beta \in \liep$. Then $\lambda(x,\beta) <+\infty$, if and only if $E(c_x^\beta ) < \infty$. Moreover, if $E(c_x^\beta ) < \infty$, denoting by $y=\lim_{t\mapsto+\infty} \exp(t\beta)x$, then $y\in X^\beta$ and
$
\lambda(x,\beta)=\langle \mup(y), \beta \rangle.
$
\end{proposition}
\begin{proof}
Since
\[
\lambda(x,\beta)=\langle \mup(z),\beta \rangle +E(c_x^\beta),
\]
it follows that $\lambda(x,\beta) < +\infty$ if and only if $E(c_x^\beta)<+\infty$. Assume that $E(c_x^\beta)<+\infty$. Let $y=\lim_{t\mapsto+\infty} \exp(t\beta)x$. Since $\beta_X$ is the gradient of $\mup^\beta$, it follows that $y\in X^\beta$ and
\[
\lambda(x,\beta)=\lim_{t\mapsto +\infty} \langle \mup(\exp(t\beta) x),\beta \rangle =\langle \mup(y),\beta\rangle.
\]
\end{proof}
We claim that the maximal weight satisfies a $G$-equivariant property. We start with the following Lemmata.
\begin{lemma}\label{inv1}
Let $c_x^\beta$ and let $g\in G^{\beta-}$. If $E(c_x^\beta)<+\infty$, then $E(c_{gx}^\beta)<+\infty$.
\end{lemma}
\begin{proof}
Since $E(c_x^\beta)<+\infty$, it follows that $\lim_{t\mapsto +\infty} \exp(t \beta)x$ exists. The element $g\in G^{\beta-}$ and so the $\lim_{t\mapsto +\infty} \exp(t\beta)\, g\exp(-t\beta)$ exists. This implies that
\[
\lim_{t\mapsto +\infty} \exp(t\beta)gx=\lim_{t\mapsto +\infty}  \bigl(\exp(t\beta)\, g\exp(-t\beta) \bigr)\exp(t\beta )x\, \mathrm{ exists}.
\]
By the definition of the maximal weight, it follows that $\lambda(gx,\beta)<+\infty$. By Proposition \ref{energy-complete}, we get $E(c_{gx}^\beta)<+\infty$,
concluding the proof.
\end{proof}
\begin{lemma}\label{inv2}
Let $g\in G^{\beta-}$. Then $\lambda(gx,\beta)=\lambda(x,\beta)$.
\end{lemma}
\begin{proof}
Assume $E(c_x^\beta)<+\infty$. Then $\lim_{t\mapsto +\infty} \exp(t\beta)x=y$ and $y\in X^\beta$. Since $g\in G^{-\beta}$, it follows that $\lim_{t\mapsto+\infty} \exp(t\beta)\, g \exp(-t\beta)$ exists and it belongs to $G^{\beta}$. Hence
\[
\lim_{t\mapsto +\infty} \exp(t\beta)gx=\lim_{t\mapsto +\infty}   \bigl(\exp(t\beta)\, g\exp(-t\beta) \bigr)\exp(t\beta) x =g_o y,
\]
for some $g_o \in G^\beta$. By Proposition \ref{invariance-critical points}, it follows that
\[
\lambda(x,\beta)=\langle \mup(y), \beta \rangle =\langle \mup(g_o y), \beta \rangle=\lambda(gx,\beta).
\]
If $E(c_x^\beta)=\infty$, then $\lambda(x,\beta)=+\infty$. By Lemma \ref{inv1}, $E(c_{gx}^\beta)=+\infty$ as well and so $\lambda(gx,\beta)=+\infty$.
\end{proof}
Let $g\in G$. By Proposition \ref{parabolic-decomposition}, $g=kh$, where $h\in G^{\beta-}$ and $k\in K$. The following two propositions are similar to the results stated in \cite[Propostion 2.11]{Teleman}.
\begin{proposition}\label{maximal-weight-equivariance}
In the above assumption, we have
\[
\lambda(gx,\beta)=\lambda(x,\mathrm{Ad}(k^{-1})(\beta)).
\]
\end{proposition}
\begin{proof}
By the above Lemma, it is enough to prove that  $\lambda(kx,\beta)=\lambda(x,\mathrm{Ad}(k^{-1})(\beta))$. Since $\mup(\exp(t\beta )kx)=k\mup(\exp(t\mathrm{Ad}(k^{-1})(\beta)))$, it follows
\[
\lambda(kx,\beta)=\lim_{t\mapsto +\infty} \langle \mup(\exp(t\mathrm{Ad}(k^{-1})(\beta)x), \mathrm{Ad}(k^{-1})(\beta)\rangle=\lambda (x,\mathrm{Ad}(k^{-1})(\beta)).
\]
\end{proof}
\begin{proposition}\label{Energy-complete}
If $(x_n,\beta_n)_n$ converges to $(x,\beta)$ then $\lambda(x,\beta) \leq \liminf_{n\to \infty}\lambda(x_n, \beta_n).$
\end{proposition}
\begin{proof}
We prove by contradiction. If the statement was false there would exist $\epsilon > 0$ and a subsequence $(x_{n_m}, \beta_{n_m})_m$ of $(x_n, \beta_n)_n$ such that the limit $\lim_{m\to\infty}\lambda(x_{n_m}, \beta_{n_m})$ exists, finite and $\lambda(x,\beta)\geq \lim_{m\to +\infty}\lambda(x_{n_m}, \beta_{n_m}) + \epsilon.$ We can choose sufficiently large $t$ such that $\lambda(x,\beta,t)\geq \lim_{m\to +\infty}\lambda(x_{n_m}, \beta_{n_m}) + \frac{\epsilon}{2}.$ However, since $\lambda(x_{n_m}, \beta_{n_m}) \geq \lambda(x_{n_m}, \beta_{n_m},t),$ (because $\lambda(x,\beta,t)$ is an increasing function) and $(x, \beta) \mapsto \lambda(x,\beta,t)$ is continuous on $X\times \liep,$ we have
 $$
 \lim_{m\to +\infty}\lambda(x_{n_m}, \beta_{n_m})\geq \lim_{m\to +\infty}\lambda(x_{n_m}, \beta_{n_m}, t) = \lambda(x,\beta,t) \geq \lim_{m\to +\infty}\lambda(x_{n_m}, \beta_{n_m}) + \frac{\epsilon}{2}.
 $$ which is not possible.
\end{proof}
The following result is the moment weight inequality which is an important ingredient to prove our results. We give a proof of this inequality applying the idea in \cite{Salamon,Chen} which was due to Xiuxiong Chen to our case.

\begin{theorem}\label{GMWI}(Moment-Weight Inequality). For every $x\in X,$ $\beta\in \mathfrak{p}\backslash\{0\}$ and $g\in G,$
$$\frac{-\la(x,\beta)}{\parallel \beta\parallel} \leq \parallel\mu_\mathfrak{p}(gx)\parallel.$$
\end{theorem}
\begin{proof}
Let $x\in X$ and $\beta\in \mathfrak{p}\backslash\{0\}$, and $g\in G.$ If $\lambda(x,\beta)=+\infty$, then the result follows. Assume that $\lambda(x,\beta)<+\infty$. By Proposition \ref{energy-complete}, $\lim_{t\mapsto +\infty} \exp(t\beta)x$ exists.

Define $\alpha : [0, \infty) \to \mathfrak{p}$ and $k : [0, \infty) \to K$ such that \begin{equation}\label{ete}
\exp(\alpha(t))k(t) = \exp(t\beta)g^{-1}.
\end{equation}
We claim that
\[
\lim_{t\mapsto +\infty} \frac{\alpha(t)}{\parallel \alpha (t) \parallel}=\frac{\beta}{\parallel \beta \parallel}.
\]
To prove this, observe that $\exp(-\alpha(t))\exp(t\beta)g^{-1}\in K.$ Then from Lemma C.2 in \cite{Salamon}, there is a positive constant $c$ such that $\parallel t\beta - \alpha(t)\parallel \leq c$ for all $t\in [0, \infty).$ Therefore, for any $t\in (0,+\infty)$, we have
\begin{align*}
    \left\vert\left\vert \frac{\beta}{\parallel \beta\parallel} - \frac{\alpha(t)}{\parallel \alpha(t)\parallel} \right\vert\right\vert &= \left\vert\left\vert\frac{t\beta}{t\parallel\beta\parallel} -\frac{\alpha(t)}{t\parallel\beta\parallel} + \frac{\alpha(t)}{t\parallel\beta\parallel} - \frac{\alpha(t)}{\parallel\alpha(t)\parallel} \right\vert\right\vert\\
    &\leq \frac{\parallel t\beta - \alpha(t)\parallel}{t\parallel\beta\parallel} + \parallel\alpha(t)\parallel\left| \frac{1}{t\parallel\beta\parallel} -  \frac{1}{\parallel\alpha(t)\parallel} \right|\\
    & = \frac{\parallel t\beta - \alpha(t)\parallel}{t\parallel \beta\parallel} +  \frac{|\parallel t\beta\parallel - \parallel \alpha(t)\parallel|}{t\parallel \beta\parallel}\\
    &\leq \frac{2c}{t\parallel\beta \parallel}.
\end{align*}
Therefore $$\lim_{t\to +\infty}\frac{\alpha(t)}{\parallel \alpha(t)\parallel} = \frac{\beta}{\parallel \beta\parallel}.$$ 
For any $t\in (0,+\infty)$, applying Lemma \ref{increasing}, the function $$s\mapsto g(s)= \langle \mu_\mathfrak{p}(\exp\left(sk^{-1} (t)\alpha(t) k(t)\right) gx), k^{-1}(t)\alpha(t) k(t)\rangle$$ is nondecreasing. In particular
\[
g(0)=\langle \mup(gx),k^{-1}(t)\alpha(t) k(t) \rangle \leq g(1)= \langle \mu_\mathfrak{p}\left( \exp (k^{-1} (t)\alpha(t) k(t)) gx \right), k^{-1}(t)\alpha(t) k(t)\rangle.
\]
Therefore, keeping in mind that $\exp(k^{-1}(t)\alpha(t) k(t))=k^{-1}(t)\exp(t\beta)$ and $\mup$ is $K$-equivariant, we have
\begin{align*}
    -\parallel \mu_\mathfrak{p}(gx)\parallel&\leq \parallel \alpha(t) \parallel^{-1}\langle \mu_\mathfrak{p}(gx), k^{-1}(t) \alpha(t) k(t) \rangle\\
    &\leq \parallel \alpha (t) \parallel^{-1}\langle \mu_\mathfrak{p}(\exp(k^{-1}(t)\alpha(t) k(t))gx), k^{-1}(t)\alpha (t) k(t)\rangle \\
    &= \parallel \alpha (t) \parallel^{-1}\langle \mu_\mathfrak{p}(k^{-1}(t)\exp(t\beta)x), k^{-1}(t)\alpha(t) k(t)\rangle\\
    & = \langle \mu_\mathfrak{p}(\exp(t\beta)x), \frac{\alpha(t)}{\parallel \alpha(t) \parallel} \rangle\\
    & = \left\langle \mu_\mathfrak{p}(\exp(t\beta)x), \frac{\alpha(t)}{\parallel \alpha (t)\parallel}-\frac{\beta}{\parallel \beta\parallel}+\frac{\beta}{\parallel \beta\parallel} \right\rangle\\
    & = \parallel \beta\parallel^{-1}\langle \mu_\mathfrak{p}(\exp(t\beta)x), \beta\rangle + \left\langle \mu_\mathfrak{p}(\exp(t\beta)x), \frac{\alpha (t)}{\parallel\alpha(t)\parallel}-\frac{\beta}{\parallel\beta\parallel} \right\rangle.
\end{align*}
Since $\lim_{t\mapsto+\infty} \exp(t\beta)x$ exists,
taking the limit $t\to +\infty,$ we have
$$
-\parallel \mu_\mathfrak{p}(gx)\parallel\leq\lim_{t\to +\infty}\frac{\langle \mu_\mathfrak{p}(\exp(t\beta)x), \beta\rangle}{\parallel \beta\parallel} = \frac{\la(x,\beta)}{\parallel\beta\parallel}.
$$ Hence, $$\frac{-\la(x,\beta)}{\parallel \beta\parallel} \leq \parallel \mu_\mathfrak{p}(gx)\parallel$$

\end{proof}
As an application of the general moment-weight inequality, we have the following:
\begin{theorem}\label{symplectic-Analytic}
If $x\in X$ is semistable, then  $\la(x,\beta)\geq 0$ for any $\beta\in\mathfrak{p}$
\end{theorem}
\begin{proof}
If $x\in X$ is semistable, then $\text{inf}_G\parallel \mu_\mathfrak{p}(gx)\parallel = 0$. Suppose by contradiction that there exists a $\beta\in\mathfrak{p}\backslash\{0\}$ such that $\la(x,\beta) < 0.$ Then by Theorem \ref{GMWI}, for any $g\in G$, we have
$$0< \frac{-\la(x,\beta)}{\parallel \beta\parallel} \leq \parallel\mu_\mathfrak{p}(gx)\parallel$$ and so $\text{inf}_G \parallel \mu_\mathfrak{p}(gx)\parallel > 0.$ which is a contradiction. Therefore $\la(x,\beta)\geq 0$ for any $\beta\in\mathfrak{p}.$
\end{proof}
\section{Analytic semistability and Polystability}\label{final-section}
\begin{definition}
A point $x\in X$ is called:
\begin{enumerate}
\item analytically stable if $\lambda(x,\beta) >0$ for any $\beta \in \liep \backslash\{0\}$;
\item analytically semi-stable if $\lambda(x,\beta) \geq 0$ for any $\beta \in \liep$;
\item analytically polystable if $\lambda(x,\beta) \geq 0$ for any $\beta \in \liep$ and the condition $\lambda(x,\beta)=0$ holds if and only if $\lim_{t\mapsto +\infty} \exp(t\beta) \in G\cdot x$.
\end{enumerate}
\end{definition}
Theorem \ref{stable} proves that $x$ is stable if and only if  $x$ is analytically stable. From now on, we assume that the $G$-action on $X$ is energy complete. What follows are the Hilbert numerical criteria for polystability and semistability under the assumption that $X$ is energy complete. We begin with the following Lemmata. 
\begin{lemma}\label{linearization2}
Let $x\in X$ and $\beta, \alpha \in \liep$ be such that $[\beta, \alpha]= 0.$ Suppose that limits $y:= \lim_{t\to +\infty}\exp(t\beta)x$, $z:= \lim_{t\to +\infty}\exp(t\alpha)y$ exists. By the energy completeness property the lemma will be needed only when these limits exist.
Then there exists $\delta>0$ such that  for $0<\epsilon<\delta$, we have
$$
\lim_{t\to +\infty}\exp(t(\beta + \epsilon\alpha))x = z.
$$
\end{lemma}
\begin{proof}
Fix $x\in X.$ Then $y\in X^\beta$ and $z\in X^\alpha\cap X^\beta.$ Let $\mathfrak{a} = \text{span}(\alpha, \beta)$ and $A = \exp(\mathfrak{a}).$  Since the exponential map is a diffeomorphism restricted on $\liep$, it follows that $A$ is a closed and compatible subgroup of $G$. Then $z$ is fixed by $A$ and by the linearization theorem, Corollary \ref{slice-cor-2}, there exists $A$-invariant open subsets $\Omega \subset X$ and $S\subset T_z X$ and a $A$-equivariant diffeomorphism $\phi : S \to \Omega$ such that $0\in S,$ $z\in \Omega$, $\phi(0) = z,$ $d\phi_0 = id_{T_z X}.$ Since $z= \lim_{t\to +\infty}\exp(t\alpha)y$, there is $t_0$ such that $\exp(t_0\alpha)y\in \Omega.$ Since $\Omega$ is $A$-invariant, we get $y\in\Omega$ and also, $x\in\Omega.$ Thus we can study all the limits in the linearization $S$. Hence, keeping in mind Proposition \ref{linearization1}, we may assume that $\Omega=\R^n$, $\alpha,\beta$ are symmetric matrices of order $n$ satisfying $[\alpha,\beta] = 0.$ From now on, the proof is similar to one given in \cite[pag. 1036]{BT}. For sake of completeness we give the proof.

The matrices $\alpha$ and $\beta$ are simultaneously diagonalizable. Decompose $V=\bigoplus_{\lambda \in \mathrm{Spec}(\beta)} V_{\lambda}$. This means $\beta_{|_{V_{\lambda}}}=\lambda_k \mathrm{Id}_{V_{\lambda_k}}$. Since $\lim_{t\mapsto +\infty} \exp(t\beta)x=y$, it follows that $x=v_0+v_1$, where $v_0 \in V_0$ and $v_1$ is the sum of some eigenvetors corresponding to negative eigenvalues. Therefore $\lim_{t\mapsto +\infty} \exp(t\beta)x=v_0+  \lim_{t\mapsto +\infty} \exp(t\beta)v_1$ and $\lim_{t\mapsto +\infty} \exp(t\beta)v_1=0$. This implies $v_0=y$. 

Let
\[
\delta=\mathrm{min} \left\{ \frac{-\lambda}{2|\mu|}:\, \lambda \in \mathrm{Spec}(\beta) \cap (0,-\infty),\, \mu \in \mathrm{Spec}(\alpha)\backslash \{0\} \right\}.
\]
If $\lambda \in \mathrm{Spec}(\beta) \cap (0,-\infty)$, then
\[
V_{\lambda}=W_0 \cap V_\lambda \bigoplus_{\mu \in \mathrm{Spec}(\alpha)\backslash\{0\}} (V_\lambda \cap W_\mu),
\]
where $W_0=\mathrm{Ker}\, \alpha$ and $\alpha_{|_{W_{\mu}}}=\mu \mathrm{Id}_{W_{\mu}}$.
Let $\epsilon <\delta $ and let $v\in V_\lambda$. Then $v=w_0+\sum_{\mu \in \mathrm{Spec}(\alpha)\backslash \{0\}} w_\mu$ and so
\[
(\alpha+\epsilon \beta)v=\lambda w_0+\sum_{\mu \in \mathrm{Spec}(\alpha)\backslash\{0\}} (\lambda+\epsilon \mu)w_\mu
\]
with $\lambda+\epsilon \mu <0$ for every $\mu \in \mathrm{Spec}(\alpha)\backslash\{0\}$. Therefore $\lim_{t\mapsto+\infty} \exp(t(\beta+\epsilon\alpha))v=0$. This holds for any $\lambda \in \mathrm{Spec}(\beta) \cap (-\infty,0)$. Now, keeping in mind that $x=y+v_1$, where $v_1$ is the sum of eigenvectors of $\beta$ associated to negative eigenvalues, we have
\[
\begin{split}
\lim_{t\mapsto+\infty}
\exp(t(\beta+\epsilon \alpha))x
&=\lim_{t\mapsto+\infty} \exp(t(\beta+\epsilon \alpha))(y+v_1) \\
&=\lim_{t\mapsto+\infty} \exp(t\epsilon \alpha))y+\lim_{t\mapsto+\infty} \exp(t(\beta+\epsilon \alpha))v_1\\
&=\lim_{t\mapsto+\infty} \exp(t\alpha))y \\
&=z.
\end{split}
\]
\end{proof}
\begin{lemma}\label{criterio1}
Let $x\in X$ be an analytically semistable point. Let $\beta \in \liep$ be such that $\lambda(x,\beta)=0$. Let $y=\lim_{t\mapsto +\infty} \exp(t\beta )x$. Then $\lambda(y,\alpha)\geq 0$ for any $\alpha \in \liep^\beta$.
\end{lemma}
\begin{proof}
Suppose by contradiction, there exists $\alpha\in \liep^\beta$ with $\lambda(y,\alpha) < 0.$ By Proposition \ref{energy-complete},
$$
\lim_{t\to +\infty}\exp(t\alpha)y = z
$$ exists. Let $A = \exp(\text{span}(\beta, \alpha))$, $[\beta, \alpha] = 0$ by the choice of $\alpha$. Since $y\in X^\beta$ and the flow $\exp(t\alpha$ preserves $X^\beta$, it follows that $z\in X^A$, where $X^A = \{x\in X: A\cdot x= x\}.$

By Lemma \ref{linearization2}, for all sufficiently small $\epsilon > 0,$
$$
\lim_{t\to +\infty}\exp(t(\beta + \epsilon\alpha))x = z.
$$
Hence,
\begin{align*}
\lambda(x, \beta + \epsilon\alpha) &= \lim_{t\to +\infty}\langle \mu_\liep(\exp(t(\beta + \epsilon\alpha)x)), \beta + \epsilon\alpha\rangle \\
&= \langle \mu_\liep(z), \beta + \epsilon\alpha\rangle\\
&= \mu_\liep^\beta(z) + \epsilon\langle \mu_\liep(z),\alpha\rangle\\
&= \mu_\liep^\beta(z) + \epsilon\lim_{t\to +\infty}\langle \mu_\liep(\exp(t\alpha)y),\alpha\rangle = \mu_\liep^\beta(z) + \epsilon\lambda(y,\alpha).
\end{align*} But by the choice of $\beta,$
$$
0 = \lambda(x,\beta) = \lim_{t\to +\infty}\langle \mu_\liep(\exp(t\beta)x), \beta\rangle = \langle \mu_\liep(y), \beta\rangle = \mu_\liep^\beta(y)
$$ and $\mu_\liep^\beta$ is constant along the curve $\exp{(t\alpha)y}$. This implies $\mu_\liep^\beta(z) = 0$. Hence,
$$
\lambda(x, \beta + \epsilon\alpha) = \epsilon\lambda(y,\alpha) < 0,
$$ which contradicts the analytic semistability of $x$.
\end{proof}
\begin{lemma}\label{criterio2}
Let $x\in X^\beta$. If $x$ is $G^\beta$-semistable, polystable or stable then $x$ is $G$-semistable.
\end{lemma}
\begin{proof}
Let $x\in X^\beta$. Assume  $\overline{G^{\beta}\cdot x}\cap \mu_{\liep^\beta}^{-1}(0) \neq \emptyset$.  Then there exists a sequence $g_n \in G^\beta$ such that $\mup(g_nx) \mapsto 0$. By Proposition \ref{gradient-restriction} it follows $\mup(g_n x)=\mu_{\mathfrak p^\beta} (g_n x) \mapsto 0$ and so the result follows.
\end{proof}
\begin{theorem}\label{semistable}
 Let $x\in X$. The following conditions are equivalent:
\begin{itemize}
\item[(1)]  $x$ is semistable.
\item[(2)]  $\mathrm{Inf}_G \parallel \mu(gx) \parallel=0$.
\item[(3)]  $x$ is analytically semistable.
\end{itemize}
\end{theorem}
\begin{proof}
$(1) \Rightarrow (2)$ is obvious. \\
$(2) \Rightarrow (3)$ follows by Theorem \ref{symplectic-Analytic}. \\
$(3) \Rightarrow (1)$. If $\la(x,\beta) > 0$ for any $\beta \in \liep\backslash\{0\}$, then by Theorem \ref{stable}, $x$  is stable and hence semistable. Assume there exists
$\beta\in\mathfrak{p}\backslash\{0\}$  such that $\la(x,\beta) = 0.$ By Proposition \ref{energy-complete}
$$\lim_{t\to +\infty}\exp(t\beta)x = y,$$ $y\in \overline{G\cdot x}$ and $\beta_X(y) = 0$.

Let $X^\beta = \{z\in X : \beta_X(z) = 0\}.$ $X^\beta$ is disjoint, union of closed submanifold of $X.$ By Proposition \ref{invariance-critical points}, $G^\beta$  preserves $X^\beta$. By Lemma \ref{criterio2} if $y$ is $(G^{\beta})^o$ stable, then $y$ is semistable and so $x$ is semistable as well.

Let $Y$ be the connected component of $X^\beta$ containing $y$. Now, $\mathfrak{g}^\beta = \mathfrak{k}^\beta \oplus \mathfrak{p}^\beta$ and $K^\beta$ preserves $\liep^\beta$. Since $K^\beta \cdot \beta = \beta$, it follows that we can write $\mathfrak{p}^\beta = \text{span}(\beta) \oplus \mathfrak{p}^{'}$, where $\mathfrak{p}^{'}$ is a $K^\beta$-invariant subspace of $\liep^\beta$. Now, keeping in mind that $
(G^\beta)^o =Z((G^\beta)^o)^o (G^\beta)^o_{ss}$, $Z((G^\beta)^o)^o$ is compatible and $\exp(t\beta) \in Z((G^{\beta})^o)^o$, it follows that
\[
(G^{\beta})^o=\exp(\R\beta) H,
\]
where $H$ is a closed, connected and compatible Lie group of $(G^\beta)^o$ with Lie algebra $\mathfrak{h} = \mathfrak{k}^\beta \oplus \mathfrak{p}^{'}.$ In particular,
\[
H=(K^\beta)^o \exp(\liep')
\]
and $\mathfrak{g}^\beta = \text{span}(\beta) \oplus \mathfrak{h}.$ Consider the $H$-action on $Y.$ By Lemma \ref{criterio1}, $\lambda(y,\beta'')\geq 0$ for every $\beta''\in \liep^\beta$. We separate the two cases:
\begin{enumerate}
    \item $\la(y,\beta') > 0, \forall \beta' \in \liep'\backslash\{0\}.$
    \item There exists $\beta_1 \in \liep'$ such that $\la(y,\beta_1) = 0.$
\end{enumerate}
Assume (a) holds.  We claim that $y$ is stable with respect to $(G^\beta)^\circ$.

Let $\Phi_y : \mathfrak{p}^\beta \to \mathbb{R};$ $\xi \mapsto \Phi(y, \exp(\xi))$ be the associated Kempf-Ness function. By Lemma \ref{kempfNess properness}, $\Phi_y$ is linearly proper on $\mathfrak{p}^{'}$. This implies that $\Phi(y, \exp(.))$ is bounded from below on $\mathfrak{p}^{'}$. Let
$$
m = \text{inf}_{\xi \in \mathfrak{p}^{'}}\Phi(y, \exp(\xi)).
$$
We claim that
\begin{equation}\label{claim1}
m = \text{inf}_{\xi \in \mathfrak{p}^{\beta}}\Phi(y, \exp(\xi)).
\end{equation} Indeed, $\xi\in \liep^\beta$ can be written as $\xi = \xi_1 + \xi_2;$ $\xi_1 \in \spam(\beta),$ $\xi_2 \in \mathfrak{p}^{'},$ $[\xi_1, \xi_2] = 0.$ By the cocycle condition of the Kempf-Ness function, keeping in mind that $\exp(\xi_1)y=y$, we have
$$
\Phi(y, \exp(\xi)) = \Phi(y, \exp(\xi_2 + \xi_1)) = \Phi(y, \exp(\xi_2)\exp(\xi_1)) = \Phi(y, \exp(\xi_1)) + \Phi(y, \exp(\xi_2)).
$$
We claim that $\Phi(y, \exp(\xi_1))=0$. Indeed, let $s(t)=\Phi(y,\exp(t\beta))$. Applying again the cocycle condition, keeping in mind that $\exp(t\beta)y=y$, one can check that $s(t)$ is a linear function. Therefore $s(t)=at$, for some $a\in \R$.  On the other hand
\[
0=\lambda(x,\beta)=\langle \mup(y),\beta\rangle=\lambda(y,\beta)=\lim_{t \mapsto +\infty} \desudt \Phi(y,\exp(t\beta))=a.
\]
Therefore,  $$
\Phi(y, \exp(\xi)) = \Phi(y,\exp(\xi_2)).
$$ This proves (\ref{claim1}). This means $\Phi_y$ has a critical point. By Proposition \ref{polystable-I}, there exists $g\in (G^\beta)^\circ $ such that $\mu_{\mathfrak{p}^\beta}(gy)=\mup(gy) = 0.$ Hence,
$$
\lim_{t\to +\infty}\exp(t\beta)gx =g\lim_{t\to +\infty}\exp(t\beta)x=gy   \in \overline{G\cdot x}\cap \mu_\mathfrak{p}^{-1}(0).
$$

Suppose (b) holds. Let $\beta_1 \in \liep'\backslash\{0\}$ be such that $\la(y, \beta_1) = 0$. Since $\liep^{\beta}=\mathrm{span}(\beta)\oplus \liep'$, it follows that $[\beta,\beta_1]=0$ and $\lia_1:=\mathrm{span}(\beta,\beta_1)$ has dimension $2$. By energy completeness,
$$
\lim_{y\to +\infty}\exp(t\beta_1) y = y_1\in Y
$$ exists, $(\beta_1)_X(y_1) = 0$ and $y_1\in \overline{G\cdot y}.$ Let $Y_1$ be the connected component of $Y^{\beta_1}$ containing $y_1$. We may split $\liep'=\mathrm{span}(\beta_1)\oplus \liep''$ as
$(K^{\lia_1})$-modules. As before, $H_{\beta_1} = (K^{\lia_1})^o\exp(\liep'')$ is a compatible Lie subgroup of $G^{\beta}$ with Lie algebra $\mathfrak{h}_{\beta_1} = \liek^{\lia_1} \oplus \liep''$. The $H_{\beta_1}$-action on $X$ preserves $Y_1$. Hence, one can  repeat the above procedure for the $H_{\beta_1}$-action on $Y_1$. On the other hand, if $\liea\subset \liep$ is a maximal Abelian subalgebra, then dim$(\liea)$ is an invariant of the $K$-action on $\liep$ \cite{dadok,knapp-beyond}. This means that the above procedure will iterate at most dim$(\liea)$. This shows that $\overline{G\cdot x}\cap \mu_\liep^{-1}(0) \neq \emptyset.$
\end{proof}
\begin{corollary}\label{hmss}
Let  $x\in X$. Then $x$ is semistable if and only if there exists $\xi \in \liep$ and $g\in (G^{\xi})^o$ such that $\lim_{t\mapsto +\infty} \exp(t\xi) gx \in \mup^{-1}(0)$.
\end{corollary}
We now consider the polystable condition.

Let $\nu \in \lieg$. We may split $\nu=\nu_\liek + \nu_\liep \in \liek \oplus \liep$. The following Lemma is proved in \cite{Salamon} for $G=U^\C$.
\begin{lemma}\label{tecnico}
Let $x\in X$ and let $\beta \in \liep$. If $\beta_X(x)=0$, then 
\[
\langle \mup(x),\beta\rangle =\langle \mup(gx), (\mathrm{Ad}(g)(\beta))_\liep \rangle,
\]
for any $g\in G$.
\end{lemma}
\begin{proof}
If $g=k$, then the result follows from the $K$-equivariance property of the gradient map.  Hence we may assume
$g=\exp(\nu)$. Let $g(t)=\exp(t\nu)$ and let $x(t)=g(t)x$. Let $\scalo$ denote the real part of the fixed $\mathrm{Ad}(U^\C)$-invariant inner product of Euclidian type on $\liu^\C$. We claim that
\[
f(t)=\langle \mup(g(t)x), (\mathrm{Ad}(g(t))(\beta))_\liep \rangle,
\]
is constant. Let $\xi(t)=\mathrm{Ad}(g(t))(\beta)$. Then
\[
\dot{\xi}(t)=[\nu,\xi(t)],
\]
and so $\dot{\xi_\liep}(t)=[\nu,\xi_\liek (t)]$. Therefore
\[
\dot{ f} (t)=\langle (\mathrm d \mup)_{x(t)} (\nu_X (x(t)), \xi(t)_\liep \rangle + \langle \mup (x(t)), [\nu,\xi_\liek (t)]\rangle.
\]
The first term is given by
\[
\begin{split}
\mathrm d \mup^{\xi(t)_\liep} (\nu_X (x(t) )&=\mathrm d \mu^{-i\xi(t)_{\liep}} (\nu_X (x(t)) ) \\ &=\omega \big(-J \left ((\xi_{\liep} )_X (x(t)) \right), \nu_X (x(t)) \big)\\  &=\omega \big( ( \xi_{\liep} )_X (x(t)) ), J (\nu_X (x(t))) \big).
\end{split}
\]
The second term is given by
\[
\begin{split}
\langle \mup (x(t)), [\nu,\xi_\liek (t)]\rangle&=\langle i\mu(x(t)), [\nu,\xi_\liek (t)]\rangle \\
&=-\langle \mu(x(t)), [-i\nu,\xi_\liek (t)]\rangle.
\end{split}
\]
Now, $\xi_{\liek}-i\nu \in \liek \oplus i\liep \subset \liu$. Using the $U$-equivariant property of the momentum map, keeping in mind that we think the momentum map as $\liu$-valued map by means of the $\mathrm{Ad}(U)$-invariant scalar product $-\scalo$ on $\liu$,  we have
\[
\begin{split}
-\langle \mu(x(t)), [-i\nu,\xi_\liek (t)]\rangle
&=-{\dfrac {\mathrm {d} }{\mathrm {ds}}}\bigg \vert _{s=0} \langle \mu(x(t)),\mathrm{Ad}(\exp(-si\nu))(\xi_\liek(t)) \rangle \\
&=-{\dfrac {\mathrm {d} }{\mathrm {ds}}}\bigg \vert _{s=0} \langle \mu(\mathrm{Ad}(\exp(s i\nu))x(t)),\xi_\liek (t)) \rangle\\
&=\omega \big( (\xi_\liek )_X (x(t)), J\nu_X (x(t)) \big).
\end{split}
\]
Therefore, keeping in mind that $\xi_X (x(t))=(\mathrm d g(t) )_{x} (\beta_X (x))=0$, we get
\[
\begin{split}
\dot{f}(t)&=\omega \big( (\xi_{\liep} )_X (x(t)), J (\nu_X (x(t))) \big)
+ \omega \big( (\xi_\liek)_X (x(t)), J (\nu_X (x(t) ) \big) \\
&= \omega( \left ((\xi_X (x(t)) \right), J (\nu_X (x(t))) \big) \\
&=0.
\end{split}
\]
This implies $f(1)=f(0)$ and the result follows.
\end{proof}
\begin{corollary}\label{polystable1}
Let $x\in X$ be polystable. Then $\lambda(x,\beta)\geq 0$ for any $\beta \in \liep$. Moreover, $\lambda(x,\beta)=0$ if and only if $\lim_{t\mapsto+\infty} \exp(t\xi) \in G\cdot x$.
\end{corollary}
\begin{proof}
Assume that $\mup(x)=0$. By Lemma \ref{stable-polystable} $\lambda(x,\beta)\geq 0$ and $\lambda(x,\beta)=0$ if and only if $\beta_X(x)=0$ and so $\lim_{t\mapsto+\infty} \exp(t\beta)x=x$.

Assume $\lim_{t\mapsto+\infty} \exp(t\beta)x=gx$, for some $g\in G$. Then $\beta_X (gx)=0$ and
\[
\lambda(x,\beta)=\langle\mup(gx),\beta \rangle.
\]
By Lemma \ref{tecnico}, we have
\[
\lambda(x,\beta)=\langle \mup(gx),\beta \rangle=\langle \mup(x), \mathrm{Ad}(g^{-1})(\beta)_\liep \rangle=0.
\]
Let $y=gx$. Let $\beta \in \liep$. Then $g=hk$, where $h\in G^{\beta-}$ and $k\in K$. By Proposition \ref{maximal-weight-equivariance}, we have
\[
\lambda(gx,\beta)=\lambda(x,\mathrm{Ad} (k^{-1}) (\beta)) \geq 0.
\]
By the above step, $\lambda(gx,\beta)=0$ if and only if $\lim_{t\mapsto +\infty} \exp(t\mathrm{Ad}(k^{-1})(\beta)) x \in G\cdot x$.  Since
\[
\exp(t\beta)gx=\bigl(\exp(t\beta)\, h\exp(-t\beta) \bigr) k \exp(t\mathrm{Ad}(k^{-1})(\beta))x,
\]
keeping in mind that $h\in G^{\beta-}$, it follows that $\lambda(gx,\beta)=0$ if and only if $\lim_{t\mapsto +\infty} \exp(t\beta)x \in G \cdot x.$
\end{proof}
\begin{corollary}\label{polystable2}
If $x\in X$ is analytically polystable then for any $g\in G$, $gx$ is analytically polystable as well.
\end{corollary}
\begin{proof}
Let $y=gx$. Let $\beta \in \liep$. Then $g=hk$, where $h\in G^{\beta-}$ and $k\in K$. By Proposition \ref{maximal-weight-equivariance}, we have
\[
\lambda(gx,\beta)=\lambda(x,\mathrm{Ad} (k^{-1}) (\beta)) \geq 0.
\]
Therefore $\lambda(gx,\beta)=0$ if and only if $\lim_{t\mapsto +\infty} \exp(t\mathrm{Ad}(k^{-1})) x \in G\cdot x$.  Since
\[
\exp(t\beta)gx=\bigl(\exp(t\beta)\, h\exp(-t\beta) \bigr) k \exp(t\mathrm{Ad}(k^{-1})(\beta))x,
\]
we get that $\lambda(gx,\beta)=0$ if and only if $\lim_{t\mapsto +\infty} \exp(t\beta) gx \in G \cdot x.$
\end{proof}
\begin{theorem}\label{polystable}
Let $x\in X$. Then $x$ is analytically polystable if and only if $x$ is polystable.
\end{theorem}
\begin{proof}
By Corollary \ref{polystable1}, if $x\in X$ is polystable then $x$ is analytically polystable.

Assume that $x\in X$ is analytically polystable. If $\lambda(x,\beta)>0$ for any $\beta \in \liep\backslash\{0\}$, then $x$ is stable and so the result is proved. Assume that $\lambda(x,\beta)=0$ for some $\beta \in \liep\backslash\{0\}$. Then $\lim_{t\mapsto+\infty} \exp(t\beta)x=y \in G\cdot x$ and
$\beta_X(y)=0$. By Corollary \ref{polystable2}, $y$ is analytically polystable.  By Lemma \ref{criterio2}, if $y$ is $G^\beta$-stable then $x$ is $G$-polystable. As in the proof of Theorem \ref{semistable}, we may decompose $\liep^\beta=\mathrm{span}(\beta)\oplus \liep'$ as $K^\beta$-modules and we consider the compatible subgroup $H=(K^\beta)^o \exp(\liep')$ of $G^\beta$. $H$ preserves the connected component $Y$ of $X^\beta$ containing $y$. By Lemma \ref{criterio1} $\lambda(y,\beta')\geq 0$ for any $\beta'\in \liep^\beta$.
We separate the two cases:
\begin{enumerate}
    \item $\la(y,\beta') > 0, \forall \beta \in \liep' \backslash\{0\}$
    \item There exists $\beta_1 \in \liep'$ such that $\la(y,\beta_1) = 0.$
\end{enumerate}
If $(a)$ holds then as in the previous proof, $y$ is $(G^{\beta})^o$ stable and so there exists $g\in (G^{\beta})^o$ such that $\mu_{\liep^\beta}(gx)=\mup(gx)=0$. In particular,
\[
\lim_{t\mapsto +\infty} \exp(t\beta)gx\in G\cdot x \cap \mup^{-1}(0).
\]
Otherwise, there exists $\beta_1 \in \liep' \backslash\{0\}$  such that $\lambda (y, \beta_1) = 0$. Then
$$
\lim_{y\to \infty}\exp(t\beta_1)y = y_1 \in G\cdot x,
$$
$(\beta_1)_X(y_1) = 0$ and $y_1\in G\cdot x$. By Corollary \ref{polystable2}, $y_1$ is analytically polystable. Let $Y_1$ be the connected component of $Y^{\beta_1}$ containing $y_1$. Let $\liep'=\mathrm{span}(\beta_1)\oplus \liep''$ be a splitting of $K^{\lia_1}$-modules. As in the proof of Theorem \ref{semistable}, we repeat the procedure for the action of the compatible subgroup $H_{\beta_1}=(K^{\lia_1})^o\exp(\liep'')$ on $Y_1$, where $\lia_1=\mathrm{span}(\beta,\beta_1)$. Since the dimension of a maximal Abelian subalgebra contained in $\liep$ is an invariant of the $K$-action on $\liep$ \cite{dadok,knapp-beyond}, the above procedure will iterate at most dim$(\liea)$. This shows that $G\cdot x\cap \mu_\liep^{-1}(0) \neq \emptyset.$
\end{proof}
\begin{corollary}\label{hmps}
Let $x\in X$. Then $x$ is polystable if and only if there exists $\xi \in \liep$ and $g\in (G^{\xi})^o$ such that $\lim_{t\mapsto +\infty} \exp(t\xi) gx  \in G\cdot x \cap \mup^{-1}(0) $.
\end{corollary}
\section{Linear examples}\label{linear}
Let $Z=\mathrm{Hom}(\C^n,\C^{m+n})$. We consider the natural action of $\mathrm{GL}(n,\C)$ on $Z$: $(g,L):=L\circ g^{-1}$.
We fix the $\mathrm{Ad}(\mathrm{GL}(n,\C))$-invariant inner product of Euclidian type on $\mathfrak{gl}(n,\C)$ given by  $B(X,Y)=\mathrm{Tr}(XY)$.
The maximal compact subgroup $\mathrm{U}(n)$ of $\mathrm{GL}(n,\C)$ acts in a Hamiltonian fashion on $Z$ with momentum map
\[
\mu(L)=\frac{\textbf{i}}{2}\biggl(L^*\circ L -h\mathrm{Id}_{n}\biggr),\, h\in \R,
\]
see for instance \cite[p. 1043]{BT}. Therefore
\[
\mu^{-1}(0)/\mathrm{U}(n) =\left\{\begin{array}{ll} \mathrm{Gr}_n(\C^{m+n}) & h>0\\ \{0\} & h=0 \\ \emptyset & h<0 \end{array}\right. ,
\]
where $\mathrm{Gr}_n(\C^{m+n})$ denotes the Grassmannian of the $n$ dimensional subspaces of $\C^{n+m}$. Assume $h>0$. If $L$ is injective, then $L$ is polystable. Indeed, $L^* \circ L=P^2$, where $P$ is a positive Hermitian endomorphism of $\C^n$ and $S=L\circ P^{-1}$ satisfies $S^* \circ S=\mathrm{Id}_{\C^n}$. Therefore $g=\sqrt{h} P \in \mathrm{GL}(n,\C)$, $g\cdot L=\sqrt{h} S$ and so $\mu(\sqrt{h}S)=0$.  Since the stabilizer of $L$ is trivial, it follows that $L$ is stable. If $\mathrm{Ker}\, L \neq \{0\}$, then it is easy to check that $L$ is not semistable. Indeed, as in \cite[p. 1043]{BT}, let $\beta \in \textbf{i}\mathfrak{u}(n)$ be such that $V_s \subset \mathrm{Ker}\, L$, where
\[
V_s=\bigoplus_{\begin{tiny}\begin{array}{l}\lambda \in \mathrm{Spec}(\beta) \\ \ \ \ \ \lambda <0\end{array}\end{tiny}} V_\lambda,
\]
and $h\mathrm{Tr}(\beta)<0$. Then $\lambda(L,\beta)<0$.

$\mathrm{GL}(n,\R)\subset \mathrm{GL}(n,\C)$ is compatible. Indeed, $\mathrm{GL}(n,\R)=\mathrm{O}(n)\exp(\liep)$, where
$\mathrm{O}(n)=\mathrm{GL}(n,\R)\cap \mathrm{U}(n)$ and $\liep=\lieg \cap i\mathrm{u}(n)=\mathrm{Sym}(n)$, i.e., the set of the symmetric matrices of order $n$. $\mathrm{GL}(n,\R)$ leaves  $\mathrm{Hom}(\R^n,\R^{m+n})\subset \mathrm{Hom}(\C^n,\C^{m+n})$ invariant. The associated  $\mathrm{GL}(n,\R)$-gradient map is given by
\[
\mup(L)=\frac{1}{2}\biggl(-L^T\circ L +h\mathrm{Id}_{n}\biggr),\, h\in \R.
\]
Therefore,
\[
\mup^{-1}(0)/\mathrm{O}(n) =\left\{\begin{array}{ll} \mathrm{Gr}_n (\R^{m+n}) & h>0\\ \{0\} & h=0 \\ \emptyset & h<0 \end{array}\right. ,
\]
where $\mathrm{Gr}_n(\R^{m+n})$ denotes the Grassmannian of the $n$ dimensional subspaces of $\R^{m+n}$. If $h>0$, then it is easy to check that $L$ is stable if and only if $L$ is injective. As in the previous example, one can check that if $\mathrm{Ker}\, L \neq \{0\}$, then $L$ is not semistable.

The $\mathrm{GL}(n,\R)$-gradient map associated to the $\mathrm{GL}(n,\R)$ action on $\mathrm{Hom}(\C^n, \C^{m+n})$ is given by
\[
\mup(L)=\frac{1}{2}\biggl(-A +h\mathrm{Id}_{n}\biggr),\, h\in \R,
\]
where $A=\mathrm{Re}(L^*\circ L)$. Indeed, since $L^*\circ L$ is Hermitian, we have $L^*\circ L=A+\textbf{i}C$, where $A$ is a symmetric matrix and $C$ is anti-symmetric matrix. Since $\langle A,\textbf{i}C\rangle=0$, it follows that the orthogonal projection of $L^*\circ L$ onto $\liep$ is given by $A$. Assume that $h>0$. Let $g\in \mathrm{GL}(n,\R)$. Then
$(L\circ g^{-1})^* \circ (L\circ g^{-1})=(g^{-1})^T \circ (L^*\circ L) \circ g^{-1}=(g^{-1})^T A g^{-1}+\textbf{i}\left((g^{-1})^T C g^{-1}\right)$. This implies that $L$ is polystable if and only if  $\mathrm{Re} (L^*\circ L )$  is a positive-define symmetric matrix. If $\mathrm{Re} (L^*\circ L )$ is not injective, then one can check that $L$ is not semistable.
Indeed, let $\beta \in \liep$ be such that $V_s \subset \mathrm{Ker}\, \mathrm{Re}(L^* \circ L)$, where
\[
V_s=\bigoplus_{\begin{tiny}\begin{array}{l}\lambda \in \mathrm{Spec}(\beta) \\ \ \ \ \ \lambda <0\end{array}\end{tiny}} V_\lambda,
\]
and $h\mathrm{Tr}(\beta)<0$. Then $\lambda(L,\beta)<0$.
\section{Final Remark}\label{final}
Let $(Z,\omega)$ be a \Keler manifold and $U^\C$ acts holomorphically on $Z$ with a $U$-equivariant momentum map $\mu : Z \to \mathfrak{u}.$  The  stabilities conditions depend on the choice of a maximal compact subgroup $U$ of $U^\C$, a $U$-invariant K\"ahler metric and the momentum map $\mu$.  It is well-known that two maximal compact subgroups are conjugate. 
Let $G\subset U^\C$ be compatible. Let $g\in G$ and let $U'=gUg^{-1}$. Then $\omega_g=(g^{-1})^{\star} \omega$ is a K\"ahler form and $U'$ preserves $\omega_g$. Since $B$ is $\mathrm{Ad}(U^\C)$-invariant, $B$ restricted to $\liu'$, respectively $\textbf{i}\liu'$,  is negative-define, respectively positive-define.
\begin{lemma}
The $U'$-action on $(Z,\omega_g)$ is Hamiltonian with momentum map $\mu':Z \lra \liu'$, given by
$\mu'=\mathrm{Ad}(g)\circ \mu \circ g^{-1}.$
\end{lemma}
\begin{proof}
We prove that $\mu'$ is $U'$-equivariant. Let $h\in U$. Then
\[
\begin{split}
\mu'(ghg^{-1}x)&=\mathrm{Ad}(g)(\mu(hg^{-1}x))  \\
&=\mathrm{Ad}(g)\big(\mathrm{Ad}(h) (\mu(g^{-1}x)\big)\\
&=\mathrm{Ad}(ghg^{-1}) \big(\mathrm{Ad}(g)(\mu(g^{-1}x))\big)\\
&=\mathrm{Ad}(ghg^{-1}) (\mu'(x))
\end{split}
\]
Let $\xi \in \liu$. Then
\[
\begin{split}
(\mu')^{\mathrm{Ad}(g)(\xi)} (z)&=-\langle\mu'(z), \mathrm{Ad}(g)(\xi) \rangle\\
&=-\langle \mu(g^{-1} z) , \xi \rangle,
\end{split}
\]
and so
$\mathrm d (\mu')^{\mathrm{Ad}(g)(\xi))}=\mathrm d \mu^\xi \circ \mathrm d g^{-1}$. Therefore
\[
\begin{split}
\mathrm d (\mu')_z^{\mathrm{Ad}(g)(\xi))}&=\omega(\xi_Z (g^{-1}z), \mathrm d g^{-1} \cdot) \\
&=\omega( \mathrm d g^{-1} \big( (\mathrm{Ad}(g)(\xi)_Z (z) \big), \mathrm d g^{-1} \cdot) \\
&=\omega_g ( \mathrm{Ad}(g)(\xi)_Z (z), \cdot).
\end{split}
\]
\end{proof}
Let $X$ be a $G$-invariant submanifold. Let $\mathtt g$ denote the Riemannian metric induced by the K\"ahler form $\omega$.
For any $g\in G$, we have a triple $(U',(g^{-1})^{\star} \mathtt g,\mu')$, where $U'=gUg^{-1}$. Note that $G$ is also compatible with respect to the Cartan decomposition  $U^\C=U'\exp(\liep')$.  Indeed, $G=gGg^{-1}=K' \exp(\liep')$, where $K'=gKg^{-1}=G\cap U'$ and $\liep'=\mathrm{Ad}(g)(\liep)=\lieg \cap \textbf{i}\liu'$. The associated gradient map $\mu':X \lra \liep'$ is the orthogonal projection of $i\mu'$ onto $\liep'$ with respect to $B$. One can check
\[
\mu_{\liep'}=\mathrm{Ad}(g)\circ \mup \circ g^{-1}.
\]
Following Teleman \cite{Teleman}, we say that such triples define a \emph{symplectization} of the $U^\C$-action on $Z$ with respect to $G$.  A priori the concept of energy completeness condition depends on the choice of a triple. The following result shows that the notion of energy completeness does not depend on the triple  chosen.
\begin{lemma}
If $(U,\mathtt g,\mu)$ is energy complete then $(U',(g^{-1})^{\star} \mathtt g,\mu')$ is energy complete as well.
\end{lemma}
\begin{proof}
Let $\xi \in \lieg$ and let $x\in X$. We recall that $c_x^\xi$ denotes the curve $c_x^\xi(t)=\exp(t\xi)x$. Let $\xi \in \liep$. Since
\[
c_x^{\mathrm{Ad}(g)(\xi)}=g\circ c_{g^{-1}x}^{\xi},
\]
it follows that the energy of $c_x^{\mathrm{Ad}(g)(\xi)}$ with respect to $(g^{-1})^{\star} \mathtt g$ coincides with the energy of $c_{g^{-1}x}^\xi$ with respect to $\mathtt g$. Moreover, the limit
$\lim_{t\mapsto +\infty} c_x^{\mathrm{Ad}(g)(\xi)}(t)$ exists if and only if the limit $\lim_{t\mapsto +\infty} c_{g^{-1}x}^{\xi}(t)$ does.
\end{proof}
We claim that the stable, polystable and semistable conditions do not depend on the triple chosen. It is a consequence of the following easy Lemma.
\begin{lemma}
Let $x\in X$. Then $x\in \mu_{\liep'}^{-1}(0)$ if and only if $g^{-1} x\in \mup^{-1}(0)$.
\end{lemma}
\begin{proof}
$\mu_{\liep'}(x)=0$ if and only if for any $\xi \in \liep$ we have $\langle (\mu_{\liep'} (x),\mathrm{Ad}(g)(\xi)\rangle =0$ hence if and only if $\langle \mup(g^{-1}x),\xi\rangle =0$ and so the result follows.
\end{proof}
\begin{corollary}
Let $x\in X$. Then
\begin{enumerate}
\item $G\cdot x \cap \mup^{-1}(0)\neq \emptyset$ if and only if $G\cdot x \cap (\mu_{\liep'})^{-1}(0)\neq \emptyset$;
\item  $\overline{G\cdot x} \cap \mup^{-1}(0)\neq \emptyset$ if and only if $\overline{G\cdot x} \cap (\mu_{\liep'})^{-1}(0)\neq \emptyset$.
\end{enumerate}
\end{corollary}
The norm square gradient map depends on the choice of a triple $(U,\mathtt g,\mu)$. Indeed,
\[
f(x) =\frac{1}{2}\langle \mup(x),\mup(x)\rangle.
\]
Let $g\in G$ and let $(U',(g^{-1})^{\star} \mathtt g,\mu')$ another triple. We denote by $f'$ the norm square of $\mu_{\liep'}$. Since $\scalo$ is $\mathrm{Ad}(G)$-invariant, it follows that
$f'(x)=f(g^{-1}x)$. Moreover,  $x$ is a critical point of $f'$ if and only if $g^{-1} x$ is a critical point of $f$. This implies that $K\cdot \beta$ is a critical orbit of $f$ if and only if $K'\cdot \mathrm{Ad}(g)(\beta)$ is a critical orbit of $f'$.  From now on, we assume that $X$ is compact and connected. Then the negative gradient flow line of the norm square is defined in all the real line. Moreover, if $x:\R \lra X$ denotes the negative gradient flow line, then $\lim_{t\mapsto +\infty} x(t)$ exists  \cite{Properties of Gradient map}. It is a straightforward computation that the gradient $\mathrm{grad}\, f' (x)$ with respect to $(g^{-1})^{\star} \mathtt g$ is given by $(\mathrm d g)_{g^{-1}x} \big(\mathrm{grad}\, f (g^{-1} x)\big)$.  Therefore, if $x(t)$ is the negative gradient flow line of $f$, then $g(x(t))$ is the negative gradient flow line of $f'$.

Let $x_\infty=\lim_{t\mapsto +\infty} x(t)$ and let $\beta =\mup(x_\infty)$. Then $\mu_{\liep'}(gx_\infty)=\mathrm{Ad}(g)(\beta)$ and $f(x_\infty)=f'(gx_{\infty})$. By  \cite[Theorem 4.9]{Properties of Gradient map}, the stratum of associated to $K\cdot \beta$ with respect to $f$ coincides with the stratum associated to $K' \cdot \mathrm{Ad}(g)(\beta)$ with respect to $f'$. Summing up we have proved the following result.
\begin{theorem}
If $X$ is connected and compact,
then the decomposition of $X$ in strata associated to critical orbits of the norm square gradient map does not depend on the triple chosen.
\end{theorem}


\end{document}